\documentclass[12pt]{amsart}
\usepackage{amsmath, amssymb}
\usepackage{enumerate}
\usepackage[margin=3cm]{geometry}
\usepackage{xcolor}

\newcommand{\eps}{\varepsilon}
\newcommand{\seq}{\subseteq}
\newcommand{\zz}{\mathbb{Z}}

\newtheorem{theorem}{Theorem}
\newtheorem{lemma}[theorem]{Lemma}
\newtheorem{problem}[theorem]{Problem}
\newtheorem{corollary}[theorem]{Corollary}
\newtheorem{conjecture}{Conjecture}
\theoremstyle{definition}

\title[Product representation of perfect cubes]{Product representation of perfect cubes}

\author[Fleiner]{Zsigmond Gy\"orgy Fleiner}
\email{zsgyfleiner@gmail.com}
\address{ELTE E\"otv\"os Lor\'and University Faculty of Science, 1117 Budapest, P\'azm\'any P\'eter s\'et\'any 1/A, Hungary}

\author[Juh\'asz]{M\'ark Hunor Juh\'asz}
\email{markh.shepherd@gmail.com}
\address{ELTE E\"otv\"os Lor\'and University Faculty of Science, 1117 Budapest, P\'azm\'any P\'eter s\'et\'any 1/A, Hungary}

\author[K\"ov\'er]{Blanka K\"ov\'er}
\email{koverblanka@gmail.com}
\address{ELTE E\"otv\"os Lor\'and University Faculty of Science, 1117 Budapest, P\'azm\'any P\'eter s\'et\'any 1/A, Hungary}

\author[Pach]{P\'eter P\'al Pach}
\email{pach.peter@vik.bme.hu}
\address{Department of Computer Science and Information Theory, Budapest University of Technology and Economics, M\H{u}egyetem rkp. 3., H-1111 Budapest, Hungary; \newline \hspace*{4mm} 
MTA-BME Lend\"ulet Arithmetic Combinatorics Research Group, M\H{u}egyetem rkp. 3., H-1111 Budapest, Hungary.}

\author[S\'andor]{Csaba S\'andor}
\email{sandor.csaba@ttk.bme.hu}
\address{Department of Stochastics, Institute of Mathemetics, Budapest University of Technology and Economics, M\H{u}egyetem rkp. 3., H-1111 Budapest, Hungary; \newline \hspace*{4mm} 
Department of Computer Science and Information Theory, Budapest University of Technology and Economics, M\H{u}egyetem rkp. 3., H-1111 Budapest, Hungary; \newline \hspace*{4mm} 
MTA-BME Lend\"ulet Arithmetic Combinatorics Research Group, M\H{u}egyetem rkp. 3., H-1111 Budapest, Hungary.}

\begin{document}

\maketitle

\begin{abstract}
Let  $F_{k,d}(n)$ be the maximal size of a set ${A}\subseteq [n]$ such that the equation
\[a_1a_2\dots a_k=x^d, \; a_1<a_2<\ldots<a_k\]
has no solution with $a_1,a_2,\ldots,a_k\in {A}$ and integer $x$. Erdős, Sárközy and T. Sós studied $F_{k,2}$, and gave bounds when $k=2,3,4,6$ and also in the general case. We study the problem for $d=3$, and provide bounds for $k=2,3,4,6$ and $9$, furthermore, in the general case, as well. In particular, we refute an 18 years old conjecture of Verstra\"ete.

We also introduce another function $f_{k,d}$ closely related to $F_{k,d}$: While the original problem requires $a_1, \ldots , a_k$ to all be distinct, we can relax this and only require that the multiset of the $a_i$'s cannot be partitioned into $d$-tuples where each $d$-tuple consists of $d$ copies of the same number.
\end{abstract}

\section{Introduction}
The problem of the solvability of equations of the form 
\begin{equation*}%\label{F-2-property}
a_1a_2\dots a_k=x^2, \; a_1<a_2<\ldots<a_k
\end{equation*}
in a set $A \seq [n]$ first appeared in a 1995 paper of Erdős, Sárközy and T. Sós ~\cite{ESS}. 
They investigated the maximal size of a set $A$ such that the equation cannot be solved in $A$, that is, there are no distinct $a_1, \ldots, a_k \in A$ whose product is a perfect square. This motivates the following definitions:

Let $F_{k,d}(n)$ be the maximal size of a set $A \seq [n]$ such that
\begin{equation}\label{F-k-property}
    a_1a_2\dots a_k=x^d, \; a_1<a_2<\ldots<a_k
\end{equation}
has no solution with $a_1,a_2,\ldots,a_k\in {A}$ and integer $x$. If $k\ge 2$ and $A$ is a set of positive integers such that equation (\ref{F-k-property}) cannot be solved in $A$, then $A$ is said to have property $P_{k,d}$. $\Gamma_{k,d}$ denotes the  family of sets of positive integers which have property $P_{k,d}$.  Similarly, let $f_{k,d}(n)$ be the maximal size of a set $A \seq [n]$ such that
\begin{equation}\label{f-k-property}
    a_1a_2\dots a_k=x^d
\end{equation}
has no solution with $a_1,a_2,\ldots,a_k\in {A}$ and integer $x$, except  {\it trivial} solutions that we specify below. If we allow some of the $a_i$'s in equation (\ref{f-k-property}) to coincide, some trivial solutions do arise: It is clear, for instance, that $a_1 = \ldots = a_d$ will yield a solution to the equation $a_1\ldots a_d = x^d$. Let us call a solution trivial if the multiset of the $a_i$'s can be partitioned into $d$-tuples where each $d$-tuple consists of $d$ copies of the same number: see for example $(a_1a_1a_1)(a_2a_2a_2)(a_3a_3a_3)=x^3$ for $k=9$, $d=3$. Note that trivial solutions arise only if $d\mid k$. Let $\gamma_{k,d}$ denote the family of sets $A$ of positive integers which have the property that equation (\ref{f-k-property}) cannot be solved in $A$ with the exception of trivial solutions of this kind.  Note that $f_{k,d} \leq F_{k,d} $.

With our notation, Erdős, Sárközy and T. Sós ~\cite{ESS} proved the following results: 
\begin{theorem}[Erd\H{o}s, S\'{a}rk\"{o}zy, T. S\'{o}s] For every $\ell\in \mathbb{Z}^+$, we have
\begin{enumerate}
\item $F_{2,2}(n)=\left(\frac{6}{\pi ^2}+o(1)\right)n$;
\item $\frac{n^{3/4}}{(\log n)^{3/2}}\ll F_{4,2}(n)-\pi (n)\ll \frac{n^{3/4}}{(\log n)^{3/2}}$;
\item $\frac{n^{2/3}}{(\log n)^{4/3}}\ll F_{6,2}(n)-\left(\pi (n)+\pi \left(\frac{n}{2}\right)\right)\ll n^{7/9}\log n$;
\item $\frac{n^{\frac{2\ell}{4\ell-1}}}{(\log n)^{\frac{4\ell}{4\ell-1}}}\ll F_{4\ell,2}(n)-\pi (n)\ll \frac{n^{3/4}}{(\log n)^{3/2}}$;
\item $\frac{n^{\frac{2\ell+1}{4\ell+1}}}{(\log n)^{\frac{4\ell+2}{4\ell+1}}}\ll F_{4\ell+2,2}(n)-\left(\pi (n)+\pi \left(\frac{n}{2}\right)\right)\ll n^{7/9}\log n$.
\end{enumerate}
\end{theorem}

Later Gy\H{o}ri \cite{Gyori} and the fourth named author \cite{Pach2019} improved the upper bound for $F_{6,2}(n)-\left(\pi (n)+\pi \left(\frac{n}{2}\right)\right)$. The current best upper bound is
\[
F_{6,2}(n)-\left(\pi (n)+\pi \left(\frac{n}{2}\right)\right)\ll n^{2/3}(\log n)^{2^{1/3}-1/3+o(1)}.
\]

The fourth named author \cite{Pach15} proved the lower bound $$\frac{n^{3/5}}{(\log n)^{6/5}}\ll F_{8,2}(n)-\pi (n).$$

For general cases, the current best lower bound estimates have been proved recently by the fourth named author and Vizer \cite{PachVizer}.

\begin{theorem}[Pach, Vizer] For every $\ell\in \mathbb{Z}^+$, we have
\begin{enumerate}
\item $\frac{n^{3/5}}{(\log n)^{6/5}}\ll F_{10,2}(n)-\left(\pi (n)+\pi \left(\frac{n}{2}\right)\right)$;
\item $\frac{n^{6/11}}{(\log n)^{12/11}}\ll F_{22,2}(n)-\left(\pi (n)+\pi \left(\frac{n}{2}\right)\right)$;
\item $\frac{n^{\frac{3\ell}{6\ell-2}}}{(\log n)^{\frac{3\ell}{3\ell-1}}}\ll F_{4\ell,2}(n)-\pi (n)$;
\item $\frac{n^{\frac{3\ell}{6\ell-1}}}{(\log n)^{\frac{6\ell}{6\ell-1}}}\ll F_{8\ell+2,2}(n)-\left(\pi (n)+\pi \left(\frac{n}{2}\right)\right)$;
\item $\frac{n^{\frac{6\ell-1}{12\ell-4}}}{(\log n)^{\frac{6\ell-1}{6\ell-2}}}\ll F_{8\ell+6,2}(n)-\left(\pi (n)+\pi \left(\frac{n}{2}\right)\right)$.
\end{enumerate}
\end{theorem}

Note that the case $2\mid k$ is closely related to (generalized) multiplicative Sidon sets. However, the case $2\nmid k$ seems to be much more difficult. The following is known:

\begin{theorem}[Erd\H{o}s, S\'{a}rk\"{o}zy, T. S\'{o}s]\label{thm-ess-odd}
For every $\ell\in \mathbb{Z}^+$ and $\varepsilon >0$, we have
\begin{enumerate}
\item $\frac{n}{(\log n)^{1+\varepsilon}}\ll n-F_{3,2}(n)\le n-f_{3,2}(n)\ll n(\log n)^{\frac{e\log 2}{2}-1+\varepsilon}$;
\item $\displaystyle \liminf_{n\to \infty }\frac{F_{2\ell+1,2}(n)}{n}\ge \log 2 =0.69\dots $;
\item $\frac{n}{(\log n)^2}\ll n-F_{2\ell+1,2}(n)$.
\end{enumerate}
\end{theorem}

So already for $F_{5,2}$ the right shape of the function is not yet determined and remains an interesting open problem to find it.

The following bounds can be proved for the functions $f_{k,2}(n)$ similarly to the proofs of the above estimates, we omit the details.

\begin{theorem} For every $\ell\in \mathbb{Z}^+$ and $\varepsilon >0$, we have
\begin{enumerate}
\item $f_{2,2}(n)=\left(\frac{6}{\pi ^2}+o(1)\right)n$;
\item $\frac{n^{3/4}}{(\log n)^{3/2}}\ll f_{4,2}(n)-\pi (n)\ll \frac{n^{3/4}}{(\log n)^{3/2}}$;
\item $\frac{n^{2/3}}{(\log n)^{4/3}}\ll f_{6,2}(n)-\pi (n)\ll n^{2/3}(\log n)^{2^{1/3}-1/3+\varepsilon }$;
\item $\frac{n^{\frac{2\ell}{4\ell-1}}}{(\log n)^{\frac{4\ell}{4\ell-1}}}\ll f_{4\ell,2}(n)-\pi (n)\ll \frac{n^{3/4}}{(\log n)^{3/2}}$;
\item $\frac{n^{\frac{2\ell+1}{4\ell+1}}}{(\log n)^{\frac{4\ell+2}{4\ell+1}}}\ll f_{4\ell+2,2}(n)-\pi (n)\ll n^{2/3}(\log n)^{2^{1/3}-1/3+\varepsilon }$.
\end{enumerate}
\end{theorem}

Based on the work of Erd\H{o}s, S\'{a}rk\"{o}zy, and T. S\'{o}s, Verstra\"ete ~\cite{Ver}
studied a similar problem: He aimed to find the maximal size of a set ${A} \seq [n]$ such that no product of $k$ distinct elements of ${A}$ is in the value set of a given polynomial $f\in \mathbb{Z}[x]$. He showed that for a certain class of polynomials the answer is $\Theta(n)$, for another class it is $\Theta(\pi(n))$, and conjectured that these are the only two possibilities: 
\begin{conjecture}\label{conj-ver}
    Let $f\in\mathbb{Z}[x]$ and let $k$ 
 be a positive integer. Then, for some constant $\rho=\rho(k,f)$ 
 depending only on $k$ and $f$, the maximal size of  
 a set ${A} \seq [n]$ such that no product of $k$ distinct elements of ${A}$ is in the value set of  $f$ is either $(\rho+o(1))n$ or $(\rho+o(1))\pi(n)$ as $n\to\infty$.
 
\end{conjecture}

For further related results, see  \cite{Pach15, Sárközy}.

We investigated the original problem in the case $d=3$, and provided bounds for both $F_{k,3}$ and $f_{k,3}$. As expected, several additional difficulties arise compared to the case $d=2$ which is also not fully resolved. To overcome these, various new ideas are needed of combinatorial and number theoretic nature.  We summarize our results below.

%%%% k=2
For $k=2$, the following bounds hold:
\begin{theorem}\label{thm-k=2}
    There exist positive constants $c_1$ and $c_2$ such that $$c_1n^{2/3}<n-F_{2,3}(n) \leq n-f_{2,3}(n)<c_2n^{{2}/{3}}.$$
\end{theorem}

%%%% k=3
For the case $k=3$ we prove that $f_{3,3}(n)/n$ converges to a constant $c_{3,3}\in (0,1)$, which we can approximate (theoretically to arbitrary precision):
\begin{theorem}\label{thm-k=3-f}
There exists a constant $0.6224\leq c_{3,3} \leq 0.6420$ such that  $$f_{3,3}(n)=(c_{3,3}+o(1))n.$$
\end{theorem}
\noindent
An analogous result holds for $F_{3,3}(n)$:
\begin{theorem}\label{thm-k=3-F}
There exists a constant $0.6919\leq C_{3,3} \leq 0.7136$ such that  $$F_{3,3}(n)=(C_{3,3}+o(1))n.$$
\end{theorem}

%%%% k=4
In the case $k=4$ we show that for large $n$, the following bounds hold. Our proofs generalize and extend ideas from \cite{ESS} used for the estimation of $F_{3,2}(n)$.
\begin{theorem}\label{thm-k=4}
Let $\eps > 0$. There exists some $n_0(\eps)$ such that for every $n \geq n_0(\eps)$ we have
$$\frac{n}{(\log n)^{2+\eps}}<n-F_{4,3}(n)\le n-f_{4,3}(n)<\frac{n}{(\log n)^{1-\frac{e\log 3}{2\sqrt{3}}-\eps}}.$$
\end{theorem}

%exponent: ~0.1379-eps

For general $k\nmid d $ we can improve the previously known best lower bound (Theorem~\ref{thm-ess-odd} (2)). Let 
$$
c_0:=\max_{\frac{1}{3}\le \alpha \le \frac{1}{2}}\left( -\log \alpha +(\log (1-\alpha )-\log \alpha )\log \alpha - \int_{\alpha}^{1-\alpha }\frac{\log (1-t)}{t}dt \right)= 0.82849\dots 
$$
Note that the maximum is attained at $\alpha=(1+\sqrt{e})^{-1}$ and $$c_0=\frac{\pi^2}{6}-\log^2(1+\sqrt{e})+\log(1+\sqrt{e})-2\text{Li}_2\left( \frac{1}{1+\sqrt{e}} \right).$$

\begin{theorem}\label{lowerbound}
For every $d\ge 2$ and $d \nmid k$, we have
$$
\liminf _{n\to \infty }\frac{f_{k,d}(n)}{n}\ge c_0=0.828\dots 
$$
\end{theorem}

\begin{theorem}\label{upperbound}
For every $d\ge 2$, $d\nmid k$, we have
$$
\frac{n}{(\log n)^d}\ll _{k,d} n-F_{k,d}(n).
$$
\end{theorem}

%%%% k=6
% The case $k=6$ is interesting as it is a multiple of $d=3$, but not 3 itself. Erdős, Sárközy and T. Sós investigated $F_{4,2}$ for the same reason. Here we obtained the following results:
For $k=6$ we obtained the following results:
\begin{theorem}\label{thm-k=6-f}
There exist positive constants $c_1$ and $c_2$ such that 
\[c_1\frac{n^{3/4}}{(\log n)^{3/2}}<f_{6,3}(n)-\pi (n)<c_2\frac{n^{3/4}}{(\log n)^{3/2}}.\]
\end{theorem}
\begin{theorem}\label{thm-k=6-F}
For $F_{6,3}(n)$ the following holds:
    \[F_{6,3}(n)=(1+o(1))\frac{n\log \log n}{\log n}.\]
\end{theorem}

Note that Theorem~\ref{thm-k=6-F} refutes Conjecture~\ref{conj-ver} of Verstra\"ete ~\cite{Ver}.

%%%% k=9
\begin{theorem}\label{thm-k=9-f}
For $f_{9,3}(n)$ we have the following bounds:
    \[ \frac{n^{2/3}}{(\log n)^{4/3}} \ll f_{9,3}(n) - \pi(n) \ll n^{2/3}\log n.    \]
    
\end{theorem}

\begin{theorem}\label{thm-k=9-F}
For $F_{9,3}(n)$ we have the following bounds:
    \[ \frac{n^{5/6}}{(\log n)^{5/3}} < F_{9,3}(n) - \left(\pi(n) + \pi\left(\frac{n}{2}\right)\right) \ll n^{5/6}.  \]
    
\end{theorem}

\begin{theorem}\label{thm-k=1233-F}
We have the following bounds:
\begin{enumerate}
\item $\frac{n^{3/4}}{(\log n)^{3/2}}\ll F_{12,3}(n)-\left(  \pi (n)+\pi(\frac{n}{2})+\pi (\frac{n}{3}) \right)\ll n^{5/6}$,
\item $\frac{n^{3/4}}{(\log n)^{3/2}}\ll F_{15,3}(n)-\left(  \pi (n)+\pi(\frac{n}{2})+\pi (\frac{n}{3})+\pi (\frac{n}{5}) \right)\ll n^{5/6}$,
\item $\frac{n^{2/3}}{(\log n)^{4/3}}\ll F_{18,3}(n)-\left(  \pi (n)+\pi(\frac{n}{2}) \right)\ll n^{5/6}$,
\item $\frac{n^{2/3}}{(\log n)^{4/3}}\ll F_{21,3}(n)-\left(  \pi (n)+\pi(\frac{n}{2})+\pi (\frac{n}{3}) \right)\ll n^{5/6}$,
\item $\frac{n^{3/5}}{(\log n)^{6/5}}\ll F_{24,3}(n)-\left(  \pi (n)+\pi(\frac{n}{2})+\pi (\frac{n}{3}) \right)\ll n^{5/6}$,
\item $\frac{n^{3/5}}{(\log n)^{6/5}}\ll F_{27,3}(n)-\left(  \pi (n)+\pi(\frac{n}{2}) \right)\ll n^{5/6}$,
\item $\frac{n^{3/5}}{(\log n)^{6/5}}\ll F_{30,3}(n)-\left(  \pi (n)+\pi(\frac{n}{2})+\pi (\frac{n}{3}) \right)\ll n^{5/6}$,
\item $\frac{n^{3/5}}{(\log n)^{6/5}}\ll F_{33,3}(n)-\left(  \pi (n)+\pi(\frac{n}{2})+\pi (\frac{n}{3}) \right)\ll n^{5/6}$.
\end{enumerate}
\end{theorem}

\begin{theorem}\label{thm-k=3l-f}
For every $\ell\ge 2$, we have
$$
\frac{n^{\frac{3\ell}{6\ell-2}}}{(\log n)^{\frac{3\ell}{3\ell-1}}}\ll f_{6\ell,3}(n)-\pi (n)\ll n^{2/3}\log n
$$
and
$$
\frac{n^{\frac{3\ell+1}{6\ell}}}{(\log n)^{\frac{3\ell+1}{3\ell}}}\ll f_{6\ell+3,3}(n)-\pi (n)\ll n^{2/3}\log n.
$$
\end{theorem}

\begin{theorem}\label{thm-k=36-F}
For every $\ell\ge 1$, we have the following bounds:
\begin{enumerate}
\item $\frac{n^{\frac{9\ell}{18\ell-2}}}{(\log n)^{\frac{9\ell}{9\ell-1}}}\ll F_{36\ell,3}(n)-\left( \pi(n)+\pi(\frac{n}{2})\right)\ll n^{5/6}$,
\item $\frac{n^{\frac{9\ell}{18\ell-2}}}{(\log n)^{\frac{9\ell}{9\ell-1}}}\ll F_{36\ell+3,3}(n)-\left( \pi(n)+\pi(\frac{n}{2})+\pi (\frac{n}{3})\right)\ll n^{5/6}$,
\item $\frac{n^{\frac{9\ell+1}{18\ell}}}{(\log n)^{\frac{9\ell+1}{9\ell}}}\ll F_{36\ell+6,3}(n)-\left( \pi(n)+\pi(\frac{n}{2})+\pi (\frac{n}{3})\right)\ll n^{5/6}$,
\item $\frac{n^{\frac{9\ell+1}{18\ell}}}{(\log n)^{\frac{9\ell+1}{9\ell}}}\ll F_{36\ell+9,3}(n)-\left( \pi(n)+\pi(\frac{n}{2})\right)\ll n^{5/6}$,
\item $\frac{n^{\frac{9\ell+3}{18\ell+4}}}{(\log n)^{\frac{9\ell+3}{9\ell+2}}}\ll F_{36\ell+12,3}(n)-\left( \pi(n)+\pi(\frac{n}{2})+\pi (\frac{n}{3})\right)\ll n^{5/6}$,
\item $\frac{n^{\frac{9\ell+3}{18\ell+4}}}{(\log n)^{\frac{9\ell+3}{9\ell+2}}}\ll F_{36\ell+15,3}(n)-\left( \pi(n)+\pi(\frac{n}{2})+\pi (\frac{n}{3})\right)\ll n^{5/6}$,
\item $\frac{n^{\frac{9\ell+4}{18\ell+6}}}{(\log n)^{\frac{9\ell+4}{9\ell+3}}}\ll F_{36\ell+18,3}(n)-\left( \pi(n)+\pi(\frac{n}{2})\right)\ll n^{5/6}$,
\item $\frac{n^{\frac{9\ell+4}{18\ell+6}}}{(\log n)^{\frac{9\ell+4}{9\ell+3}}}\ll F_{36\ell+21,3}(n)-\left( \pi(n)+\pi(\frac{n}{2})+\pi (\frac{n}{3})\right)\ll n^{5/6}$,
\item $\frac{n^{\frac{9\ell+6}{18\ell+10}}}{(\log n)^{\frac{9\ell+6}{9\ell+5}}}\ll F_{36\ell+24,3}(n)-\left( \pi(n)+\pi(\frac{n}{2})+\pi (\frac{n}{3})\right)\ll n^{5/6}$,
\item $\frac{n^{\frac{9\ell+6}{18\ell+10}}}{(\log n)^{\frac{9\ell+6}{9\ell+5}}}\ll F_{36\ell+27,3}(n)-\left( \pi(n)+\pi(\frac{n}{2})\right)\ll n^{5/6}$,
\item $\frac{n^{\frac{9\ell+7}{18\ell+12}}}{(\log n)^{\frac{9\ell+7}{9\ell+6}}}\ll F_{36\ell+30,3}(n)-\left( \pi(n)+\pi(\frac{n}{2})+\pi (\frac{n}{3})\right)\ll n^{5/6}$,
\item $\frac{n^{\frac{9\ell+7}{18\ell+12}}}{(\log n)^{\frac{9\ell+7}{9\ell+6}}}\ll F_{36\ell+33,3}(n)-\left( \pi(n)+\pi(\frac{n}{2})+\pi (\frac{n}{3})\right)\ll n^{5/6}$.
\end{enumerate}

%$$
%\frac{n^{\frac{9\ell}{18\ell-2}}}{(\log n)^{\frac{9\ell}{9\ell-1}}}\ll F_{18\ell,3}(n)-\left(\pi (n)+\pi \left(\frac{n}{2}\right)\right)\ll n^{5/6}
%$$
%and
%$$
%\frac{n^{\frac{9\ell+4}{18\ell+6}}}{(\log n)^{\frac{9\ell+4}{9\ell+3}}}\ll F_{18\ell+9,3}(n)-\left(\pi (n)+\pi \left(\frac{n}{2}\right)\right)\ll n^{5/6}.
%$$
\end{theorem}

%\begin{theorem}\label{thm-k=9l+3-F}
%For every $\ell\ge 1$, we have
%$$
%\frac{n^{\frac{9\ell+1}{18\ell}}}{(\log n)^{\frac{9\ell+1}{9\ell}}}\ll F_{18\ell+3,3}(n)-\left(\pi (n)+\pi \left(\frac{n}{2}\right)+\pi \left(\frac{n}{3}\right)\right)\ll n^{5/6}.
%$$
%and
%$$
%\frac{n^{\frac{9\ell-3}{18\ell-8}}}{(\log n)^{\frac{9\ell-3}{9\ell-4}}}\ll F_{18\ell-6,3}(n)-\left(\pi (n)+\pi \left(\frac{n}{2}\right)+\pi \left(\frac{n}{3}\right)\right)\ll n^{5/6}.
%$$
%\end{theorem}

%\begin{theorem}\label{thm-k=15-F}
%For $F_{15,3}(n)$ we have the following bounds:
%$$
%\frac{n^{3/5}}{(\log n)^{6/5}}\ll F_{15,3}(n)-\left(\pi (n)+\pi \left(\frac{n}{2}\right)+\pi \left(\frac{n}{3}\right)+\pi \left(\frac{n}{5}\right)\right)\ll n^{5/6}.
%$$
%\end{theorem}

%\begin{theorem}\label{thm-k=9l+6-F}
%For every $\ell\ge 1$, we have
%$$
%\frac{n^{\frac{9\ell+3}{18\ell+4}}}{(\log n)^{\frac{9\ell+3}{9\ell+2}}}\ll F_{18\ell+6,3}(n)-\left(\pi (n)+\pi \left(\frac{n}{2}\right)+\pi \left(\frac{n}{3}\right)\right)\ll n^{5/6}.
%$$
%and
%$$
%\frac{n^{\frac{9\ell+7}{18\ell+12}}}{(\log n)^{\frac{9\ell+7}{9\ell+6}}}\ll F_{18\ell+15,3}(n)-\left(\pi (n)+\pi \left(\frac{n}{2}\right)+\pi \left(\frac{n}{3}\right)\right)\ll n^{5/6}.
%$$
%\end{theorem}

%%%% Notations
\textbf{Notations.} Throughout this paper, we denote by $[n]$ the set $\{1,2, \ldots, n\}$. The standard notation $\ll$, $\gg$
and $O$ is applied to positive quantities in the usual way.
That is, $X \gg Y$, $ Y \ll X$, and $Y = O(X)$ all mean that $X \geq cY$,
for some absolute constant $c > 0$. %If both $X \ll Y$ and $Y \ll X $ hold, we write $X = \Theta(Y )$.
If the constant $c$ depends on a quantity
$t$, we write $X \ll_t Y$, $Y = O_t(Y )$. Analogously to squarefree numbers, we call an integer $a$ cubefree if there is no integer $b>1$ such that $b^3 \mid a$. The cubefree part of an integer $a$ is $a/b^3$, where $b^3$ is the largest perfect cube dividing $a$. In this paper by the interval $\left[a,b\right]$ we mean only the integers in $\left[a,b\right]$. We denote by $\Omega(m)$ the total number of prime factors of $m$ with multiplicity and $\pi_{k}(n)$ denotes the number of positive integers up to $n$ which have exactly $k$ prime factors (with multiplicity).

% pi_2 def

%% smooth, rough, o, O

\section{Combinatorial and arithmetic lemmas}

Let $G_1,G_2,\dots ,G_r$ be arbitrary graphs. The Tur\'{a}n number $\mathrm{ex}(n;G_1,\dots ,G_r)$ is the maximum number of edges in a graph on $n$ vertices not containing any copy of $G_1$, $G_2$, ... or $G_r$. Let $C_\ell$ denote the cycle of length $\ell$. A complete bipartite graph with partite sets of size $u$ and $v$ is denoted $K_{u,v}$.

For the proof of Theorem~\ref{thm-k=9-F} we need the following graph-theoretic lemma.

\begin{lemma}\label{k33}
    For every sufficiently large $n$, there exists a $K_{3,3}$-free bipartite graph $G=(S,T,E)$ with $|S|=|T|=n$ such that $|E|>cn^{5/3}$ for some $c > 0$.
\end{lemma}

\begin{proof}[Proof of Lemma~\ref{k33}]
Brown~\cite{Brown} proved that there exists a graph on $2n$ vertices with $c'n^{5/3}$ edges that does not contain $K_{3,3}$ as a subgraph, where $c' > 0$. Let us denote this graph by $G_0$. We use a standard probabilistic argument to show that there exists a bipartite graph $G$ that satisfies the requirements of the lemma. Let us divide the vertices of $G_0$ into two sets of equal sizes with uniform distribution, and delete every edge within the two sets. We delete every edge with probability $\frac{n-1}{2n-1}$, so the expected value of the number of remaining edges is at least $\frac{1}{2}c'n^{5/3}$. Hence, there exists such a bipartite graph with at least $\frac{1}{2}c'n^{5/3}$ edges as we claimed.
\end{proof}

The following statements will be used several times in the proofs.

\begin{lemma}\label{exgraph} We have the following estimates:
\begin{enumerate}
%\item $\mathrm{ex}(n;C_3,C_4)=(0.5+o(1))n^{3/2}$ (as $n\to \infty)$;
\item $n^{3/2}\ll \mathrm{ex}(n;C_3,C_4,C_5)\ll n^{3/2}$;
%\item $n^{4/3}\ll \mathrm{ex}(n;C_3,C_4,C_5,C_6)\ll n^{4/3}$;
\item $n^{4/3}\ll \mathrm{ex}(n;C_3,C_4,C_5,C_6,C_7)\ll n^{4/3}$;
%\item $n^{6/5}\ll \mathrm{ex}(n;C_3,C_4,\dots ,C_{10})\ll n^{6/5}$;
\item $n^{6/5}\ll \mathrm{ex}(n;C_3,C_4,\dots ,C_{11})\ll n^{6/5}$;
%\item for every $k\ge 3$, we have $\mathrm{ex}(n;C_3,C_4,\dots ,C_{4k})\gg n^{\frac{3k}{3k-1}}$;
\item for every $k\ge 3$, we have $\mathrm{ex}(n;C_3,C_4,\dots ,C_{4k+1})\gg n^{\frac{3k}{3k-1}}$;
%\item for every $k\ge 3$, we have $\mathrm{ex}(n;C_3,C_4,\dots ,C_{4k+2})\gg n^{\frac{3k+1}{3k}}$;
\item for every $k\ge 3$, we have $\mathrm{ex}(n;C_3,C_4,\dots ,C_{4k+3})\gg n^{\frac{3k+1}{3k}}$.
%\item $n^{5/3}\ll \mathrm{ex}(n,K_{3,3})\ll n^{5/3}$
\end{enumerate}
\end{lemma}

\begin{proof} 
\begin{enumerate}
    %\item See \cite{ERS}.
    \item It is known that $\mathrm{ex}(n;C_3,C_4)=(0.5+o(1))n^{3/2}$ as $n\to \infty$ (see \cite{ERS}). Similarly to the proof of Lemma~\ref{k33}, it can be proved that there exists a $\{ C_3,C_4\}$-free bipartite graph $G=(V,E)$, where $|V|=n$ and $|E|\gg n^{3/2}$. Then $G$ must be $\{ C_3,C_4,C_5\} $-free.    
    %\item This is a special case of a theorem of Bondy and Simonovits %\cite{BS}.
    \item It is known that $n^{4/3}\ll \mathrm{ex}(n;C_3,C_4,C_5,C_6)\ll n^{4/3}$ (see \cite{BS}). Similarly to the proof of Lemma~\ref{k33}, it can be proved that there exists a $\{ C_3,C_4,C_5,C_6\}$-free bipartite graph $G=(V,E)$, where $|V|=n$ and $|E|\gg n^{4/3}$. Then $G$ must be $\{ C_3,C_4,C_5,C_6,C_7\} $-free.
    %\item See \cite{Benson} or \cite{Singleton}.
    \item It is known that $n^{6/5}\ll \mathrm{ex}(n;C_3,C_4,\dots ,C_{10})\ll n^{6/5}$ (see \cite{Benson}).Similarly to the proof of Lemma~\ref{k33}, it can be proved that there exists a $\{ C_3,C_4,\dots ,C_{10}\}$-free bipartite graph $G=(V,E)$, where $|V|=n$ and $|E|\gg n^{6/5}$. Then $G$ must be $\{ C_3,C_4,\dots ,C_{11}\} $-free.
    %\item See \cite{LUW}.
    \item It is known that $\mathrm{ex}(n;C_3,C_4,\dots ,C_{4k})\gg n^{\frac{3k}{3k-1}}$ (see \cite{LUW}). Similarly to the proof of Lemma~\ref{k33}, it can be proved that there exists a $\{ C_3,C_4,\dots ,C_{4k}\}$-free bipartite graph $G=(V,E)$, where $|V|=n$ and $|E|\gg n^{\frac{3k}{3k-1}}$. Then $G$ must be $\{ C_3,C_4,\dots ,C_{4k+1}\} $-free.
    %\item See \cite{LUW}.
    \item It is known that $k\ge 3$, we have $\mathrm{ex}(n;C_3,C_4,\dots ,C_{4k+2})\gg n^{\frac{3k+1}{3k}}$ (see \cite{LUW}). Similarly to the proof of Lemma~\ref{k33}, it can be proved that there exists a $\{ C_3,C_4,\dots ,C_{4k+2}\}$-free bipartite graph $G=(V,E)$, where $|V|=n$ and $|E|\gg n^{\frac{3k+1}{3k}}$. Then $G$ must be $\{ C_3,C_4,\dots ,C_{4k+3}\} $-free.
    
    %\item This is Corollary 8.3 in \cite{Simonovits}.
    %\item See \cite{Brown}.
\end{enumerate}
\end{proof}

%Similar to the proof of Lemma~\ref{k33}, it can be proved from the previous lemma the following statements.

%\begin{lemma}\label{oddexgraph} We have the following estimates:
%\begin{enumerate}
%\item $n^{3/2}\ll \mathrm{ex}(n;C_3,C_4,C_5)\ll n^{3/2}$;
%\item $n^{4/3}\ll \mathrm{ex}(n;C_3,C_4,C_5,C_6,C_7)\ll n^{4/3}$;
%\item $n^{6/5}\ll \mathrm{ex}(n;C_3,C_4,\dots ,C_{11})\ll n^{6/5}$;
%\item for every $k\ge 2$, we have $\mathrm{ex}(n;C_3,C_4,\dots ,C_{4k+1})\gg %n^{\frac{3k}{3k-1}}$;
%\item for every $k\ge 3$, we have $\mathrm{ex}(n;C_3,C_4,\dots ,C_{4k+3})\gg n^{\frac{3k+1}{3k}}$.
%\item $n^{5/3}\ll \mathrm{ex}(n,K_{3,3})\ll n^{5/3}$
%\end{enumerate}
%\end{lemma}

Let $F$ be an $r$-uniform hypergraph. The Tur\'{a}n number $\mathrm{ex}(n,F)$ is the maximum number of edges in an $F$-free $r$-uniform hypergraph on $n$ vertices. Let $K^r_t$ be the complete $r$-uniform hypergraph on $t$ vertices. The following lemma was proved by de Caen \cite{Caen3}:
\begin{lemma}\label{Caen} For every $r\ge 2$ and $t>r$, we have
$$
\lim_{n\to \infty}\frac{\mathrm{ex}(n,K^r_t)}{\binom{n}{r}}\le 1-\frac{1}{\binom{t-1}{r-1}}.
$$
\end{lemma}

We will use the following lemma:

\begin{lemma}\label{k=6-l1}
Let $k\geq 2$ be a positive integer. 
%Let ${P}=\{ p_1,\dots ,p_t \}$ be a finite set of distinct primes, and let ${A}$ be a multiset of integers containing elements of the form $p_ip_j$ with $p_i,p_j\in {P}$. Let $G({A})$ denote the multigraph with $V=\{ P_1,\dots ,P_t\} $ and $E=\{ (P_i,P_j): p_ip_j\in {A}\} $. If $G({A})$ is $\{C_3,\dots,C_{2k}\}$-free, then ${A}\in \gamma _{3k,3}$.
Let $Q_n=\{ p_1,\dots ,p_t\}$ be the set of primes not exceeding $\sqrt{n}$ and $R_n$ be the set of primes from the interval $(\sqrt{n},n]$. 
Let $S_n^{(k)}=R_n\cup A_n^{(k)}$, where  each element of the set $A_n^{(k)}$ is the product of two different primes from $Q_n$.  Let $G(A_n^{(k)})=(V_n,E_n)$ be the graph with $V_n=\{ P_1,\dots ,P_t\}$ and $E_n=\{ (P_i,P_j): p_i,p_j\in Q_n,\  p_ip_j\in A_n^{(k)} \} $. If $G(A_n^{(k)})$ is $\{ C_3,C_4,\dots ,C_{2k}\}$-free, then $S_n^{(k)}\in \gamma_{3k,3}$.
\end{lemma}
\begin{proof}
We argue by induction on $k$. For $k=2$, the condition is that $G(A_n^{(k)})$ is $\{C_3,C_4\}$-free. For the sake of contradiction, let us assume that ${A}\notin \gamma _{6,3}$, that is there exists a nontrivial solution  $a_1a_2a_3a_4a_5a_6=x^3$ with $x\in \mathbb{Z}^+$, $a_1\le a_2\le \dots \le a_6$ and $a_i\in S_n^{(k)}$ (for every $i$). If $a_i\in R_n$ for some $1\le i\le 6$, then  $a_1=a_2=a_3$ and $a_4=a_5=a_6$, we get a trivial solution. If $a_i\notin R_n$ for every $1\le i\le 6$, then there exist distinct primes $p_1',p_2',p_3',p_4'\in Q_n$ such that $\{ a_1,\dots ,a_6 \}=\{ p_1'p_2',p_1'p_2', p_3'p_4',p_3'p_4',p_1'p_3',p_2'p_4' \}$ or $\{ a_1,\dots ,a_6\}=\{ p_1'p_2',p_1'p_3', p_1'p_4',p_2'p_3',p_2'p_4',p_3'p_4' \}$. Thus $G(A_n^{(k)})$ must contain a $C_4$, a contradiction. 

In the induction step, let us suppose that $G(A_n^{(k)})$ is $\{ C_3,\dots, C_{2k}\}$-free, but $S_n^{(k)}\notin \gamma _{3k,3}$, that is, there exists a nontrivial solution $a_1\dots a_{3k}=x^3$ with $x\in \mathbb{Z}^+$, $a_1\le a_2\le \dots \le a_{3k}$ and $a_i\in S_n^{(k)}$ (for every $i$). 

If $a_i\in R_n$ for some $i$, then the multiset $\{a_1,a_2,\dots,a_{3k}\}$ contains (at least) three copies of $a_i$, and after deleting these we are done by the induction hypothesis. Hence, it can be assumed that $a_1,\dots,a_{3k}\in A_n^{(k)}$. By the induction hypothesis we may also assume that $a_i<a_{i+2}$ for $1\le i\le 3k-2$.

Let us assign a graph $G'=(V_n,E_n')$ to the set (not to the multiset) $\{ a_1,\dots ,a_{3k}\} $ such that $(P_i,P_j)\in E_n'$ if and only if $p_ip_j=a_h$ for some $1\le h\le 3k $. The number of vertices having positive degree in $G'$ is at most $6k/3=2k$. Moreover, the condition $a_i<a_{i+2}$ (for every $1\le i\le 3k-2$) implies that $d(P_i)\ge 2$ whenever $d(P_i)>0$ in the graph $G'$. Hence, $G'$ must contain a cycle $C_{\ell}$ for some $3\le \ell\le 2k$, therefore $G(A_n^{(k)})$ also contains a cycle $C_{\ell}$, a contradiction.

\end{proof}

We need the following lemma to prove Theorem~\ref{thm-k=1233-F} and \ref{thm-k=36-F}.

\begin{lemma}\label{Gammak=6-l1}
Let $k\geq 4$ be a positive integer. 
%Let ${P}=\{ p_1,\dots ,p_t \}$ be a finite set of distinct primes, and let ${A}$ be a multiset of integers containing elements of the form $p_ip_j$ with $p_i,p_j\in {P}$. Let $G({A})$ denote the multigraph with $V=\{ P_1,\dots ,P_t\} $ and $E=\{ (P_i,P_j): p_ip_j\in {A}\} $. If $G({A})$ is $\{C_3,\dots,C_{2k}\}$-free, then ${A}\in \gamma _{3k,3}$.
Let $P$ be the set of primes. Let us suppose that the positive integers $a_1<a_2<\dots <a_{3k}$ are the products of two primes, where the canonical form of $a_1a_2\dots a_{3k}$ is $p_1^{\alpha_1}p_2^{\alpha _2}\dots p_u^{\alpha _u}$, where $p_i\in P$. Let us assign a graph $G=(V_u,E_u)$ to the set $\{ a_1,\dots ,a_{3k}\}$ such that $V_u=\{ P_1,\dots ,P_u\}$ and $E_u=\{ (P_i,P_j): p_i,p_j\in P,\  p_ip_j=a_h \text{ for some $h$} \} $.  
%Let us suppose that the set $S_n^{(k)}$ contains the product of two different primes from $Q_n$.  Let $G(S_n^{(k)})=(V_n,E_n)$ be the graph with $V_n=\{ P_1,\dots ,P_t\}$ and $E_n=\{ (P_i,P_j): p_i,p_j\in Q_n,\  p_ip_j\in S_n^{(k)} \} $. 
If $G$ is $\{ C_3,C_4,\dots ,C_{k}\}$-free, then $a_1a_2\dots a_{3k}$ is not a perfect cube.
\end{lemma}

\begin{proof}
For the sake of contradiction, let us assume that $a_1a_2\dots a_{3k}=x^3$ with $x\in \mathbb{Z}^+$. We show that $G$ contains a cycle $C_{\ell}$ with $\ell\le k$. Clearly, $d(P)\ge 3$ for every $P\in V_u$, so $G$ must contain a cycle. Moreover, the number of vertices is at most $6k/3=2k$, that is $u\le 2k$. Let us suppose that the shortest length of a cycle in the graph $G$ is $L>k$. Let us assume that the vertices of a cycle of length $L$ are $P_{i_1},\dots P_{i_L}$, $1\le i_1<\dots< i_L\le u$. Note that the shortest cycle is chordless,  thus $d(P)\ge 3$ implies that for each $P_{i_s}$ there exists a $Q_{i_s}\in V_u$ such that $(P_{i_s},Q_{i_s})\in E_u$ and $Q_{i_s}\ne P_{i_t}$ for every $1\le s,t\le L$. The conditions $L>k$ and $u\le 2k$ imply that there are integers $1\le s<t\le L$ such that $Q_{i_s}=Q_{i_t}$. Then the length of the cycle of vertices $P_{i_1},P_{i_2},\dots ,P_{i_s},Q_{i_s},P_{i_t},p_{i_{t+1}},\dots P_{i_L}$ is $s+2+L-t$ and the length of the cycle of vertices $P_{i_s},P_{i_{s+1}},\dots ,P_{i_t},Q_{i_s}$ is $t-s+2$. Hence, the length of the shortest cycle is at most $\frac{s+2+L-t+t-s+2}{2}=\frac{L}{2}+2<L$, a contradiction. 

%Hence, $G'$ must contain a cycle $C_{\ell}$ for some $3\le \ell\le k$, therefore $G(S_n^{(k)})$ also contains a cycle $C_{\ell}$, a contradiction.
\end{proof}

We need a couple of number theoretic lemmas to prove Theorem~\ref{thm-k=4}.

\begin{lemma}\label{ess-l12}
    If $1<y<2$ and $\eps>0$, then for $x>x_1(\eps)$ we have $$|\{m:\ m \leq x,\  \Omega(m) \geq y \log \log x\}| < \frac{x}{(\log x)^{1+y\log y - y - \eps}}.$$
\end{lemma}
\begin{proof}
    See \cite[Corollary 2]{ES}.
\end{proof}
\begin{lemma}\label{ess-l11}
    If $0 < y \leq 1$ and $\eps > 0$, then for $x>x_2(\eps)$ we have \[
    \frac{x}{(\log x)^{1+y\log y-y+\eps}} < |\{m:\ m \leq x,\ \Omega(m) \leq y \log \log x\}| < \frac{x}{(\log x)^{1+y\log y-y-\eps}}
    \]
\end{lemma}
\begin{proof}
    This follows from a result of Hardy and Ramanujan \cite{HR}.
\end{proof}

\begin{lemma}\label{k=4-l1}
Suppose that $a_1,a_2,a_3,a_4\in \left[\frac{n}{\log n}, n\right]$ such that $d^2|a_i$ implies $d\le \log n$. If $a_1a_2a_3a_4$ is a perfect cube, then there exist integers $u_i,v_i,w_i\in \left(\frac{\sqrt[3]{n}}{\log ^{16} n}, \sqrt[3]{n}\log ^{16} n\right)$ such that $a_i=u_iv_iw_i$.
\end{lemma}

\begin{proof}
Suppose that the conditions of the lemma hold for $a_1, a_2, a_3, a_4$, and let $a_i=b_ic_i^2d_i^3$, where $b_i$ and $c_i$ are squarefree and $\gcd(b_i,c_i)=1$. Then $c_id_i\le \log n$ and $$n\ge b_i=\frac{a_i}{c_i^2d_i^3}\ge \frac{\frac{n}{\log n}}{\log ^3 n}=\frac{n}{\log ^4 n} $$
Clearly, 
$$
\frac{n^4}{\log ^4n}\le a_1a_2a_3a_4=b_1b_2b_3b_4c_1^2c_2^2c_3^2c_4^2d_1^3d_2^3d_3^3d_4^3=x^3\le n^4
$$
for some $x\in \mathbb{Z}^+$. Let
$$
m=c_1^2c_2^2c_3^2c_4^2d_1^3d_2^3d_3^3d_4^3,
$$
where $m\le \log ^{12}n$. Let us define the positive integers $e_i$ as $e_i=\frac{b_i}{\gcd(b_i,m)}$. Since $b_i$ is squarefree we have
$$
\gcd(b_i,m)\le c_1c_2c_3c_4d_1d_2d_3d_4\leq \sqrt{m} \leq \log ^6n
$$
and
$$
n\ge e_i=\frac{b_i}{\gcd(b_i,m)}\ge \frac{\frac{n}{\log ^4n}}{\log ^6n}=\frac{n}{\log ^{10}n}.
$$
Using that $\gcd(e_i,m)=1$ and $\gcd(e_i,\gcd(b_j,m))=1$ for every $1\le i,j\le 4$ we get that
$$
\gcd(e_1e_2e_3e_4,\gcd(b_1,m)\gcd(b_2,m)\gcd(b_3,m)\gcd(b_4,m)m)=1.
$$
It follows from 
$$
a_1a_2a_3a_4=e_1e_2e_3e_4\gcd(b_1,m)\gcd(b_2,m)\gcd(b_3,m)\gcd(b_4,m)m=x^3
$$
that
$$
e_1e_2e_3e_4=y^3
$$
for some $y\in \mathbb{Z}^+$. Since $e_i$ is squarefree we have
$$
e_1=\gcd(e_1,e_2,e_3)\gcd(e_1,e_2,e_4)\gcd(e_1,e_3,e_4),
$$
$$
e_2=\gcd(e_1,e_2,e_3)\gcd(e_1,e_2,e_4)\gcd(e_2,e_3,e_4),
$$
$$
e_3=\gcd(e_1,e_2,e_3)\gcd(e_1,e_3,e_4)\gcd(e_2,e_3,e_4),
$$
and
$$
e_4=\gcd(e_1,e_2,e_4)\gcd(e_1,e_3,e_4)\gcd(e_2,e_3,e_4).
$$
Hence
$$
y^3=e_1e_2e_3e_4=\gcd(e_1,e_2,e_3)^3\gcd(e_1,e_2,e_4)^3\gcd(e_1,e_3,e_4)^3\gcd(e_2,e_3,e_4)^3,
$$
that is
$$
y=\gcd(e_1,e_2,e_3)\gcd(e_1,e_2,e_4)\gcd(e_1,e_3,e_4)\gcd(e_2,e_3,e_4).
$$
It follows from $\frac{n^4}{\log ^{40}n}\le e_1e_2e_3e_4\le n^4$  that $\frac{n^{4/3}}{\log ^{40/3}}n\le y\le n^{4/3}$. Since $\gcd(e_1,e_2,e_3)=\frac{y}{e_4}$ we get
$$
\frac{n^{1/3}}{\log ^{40/3}n}\le \gcd(e_1,e_2,e_3)\le n^{1/3}\log ^{10}n.
$$
Similarly,
$$
\frac{n^{1/3}}{\log ^{40/3}n}\le \gcd(e_1,e_2,e_4),\gcd(e_1,e_3,e_4),\gcd(e_2,e_3,e_4)\le n^{1/3}\log ^{10}n.
$$
Let 
$$
a_1=\left( \gcd(e_1,e_2,e_3)\right) \left( \gcd(e_1,e_2,e_4)\gcd(b_1,m)\right) \left( \gcd(e_1,e_3,e_4)c_1^2d_1^3\right) =u_1v_1w_1,
$$
where
$$
\frac{n^{1/3}}{\log ^{16} n}\le u_1,v_1,w_1\le n^{1/3}\log ^{16}n,
$$
which completes the proof.
\end{proof}

\begin{lemma}\label{k=4-l2}
    Let $B \seq [n]$ be the set of those positive integers $b$ that can be written in the form $b=uvw$ with $u,v,w \in \zz^+$ and $u,v,w\in \left[\frac{\sqrt[3]{n}}{\log ^{16}n}, \sqrt[3]{n}\log ^{16}n\right]$. Then for every $\eps > 0$ there exists $n_0(\eps)$ such that $$|B| \leq \frac{n}{(\log n)^{1-\frac{e\log 3}{2\sqrt{3}}-\eps }}$$ for $n \geq n_0(\eps)$.
\end{lemma}

\begin{proof}
Fix $\eps > 0$. Let $B_1$ be the set of all $b=uvw \in B$ such that  $$\min\{\Omega(u), \Omega(v), \Omega(w)\} > \frac{e}{3^{1.5}}\log \log n,$$ and let $B_2$ be the set of all $b=uvw \in B$ such that $$\Omega(w) \leq \frac{e}{3^{1.5}}\log \log n.$$ Then $B = B_1 \cup B_2$, so $|B| \leq |B_1| + |B_2|$.

Let $n \geq n_1(\eps)$ as per Lemma~\ref{ess-l12}. Now if $b \in B_1$, then $$\Omega(b)=\Omega(u)+\Omega(v)+\Omega(w) > \frac{e}{\sqrt{3}} \log \log n,$$ so it follows from Lemma~\ref{ess-l12} that  $$|B_1| \leq \frac{n}{(\log n)^{1+\frac{e}{\sqrt{3}} \log \frac{e}{\sqrt{3}} - \frac{e}{\sqrt{3}} - \eps}} = \frac{n}{(\log n)^{1-\frac{e \log 3}{2\sqrt{3}}-\eps}}.$$

For the estimation of the size of $B_2$, let 
$$S_n = \left\{(u,v):\ \frac{\sqrt[3]{n}}{(\log n)^{16}}\le u,v\le \sqrt[3]{n}(\log n)^{16}\right\}$$ 
and for each $(u,v) \in S_n$ let 
$$W_{u,v} := \left\{w:\  w \leq \frac{n}{uv},\  \Omega(w) \leq \frac{e}{3^{1.5}} \log \log n\right\},$$
$$W_{u,v}^* := \left\{w:\  w \leq \frac{n}{uv},\  \Omega(w) \leq \left(\frac{e}{3^{1.5}}+\frac{1}{\log \log n}\right) \log \log \frac{n}{uv}\right\}.$$
A simple calculation shows that, for $(u,v) \in S_n$ we have $$\log \log n - \log \log \frac{n}{uv} = \log 3 + o(1)$$
and
$$
\frac{e}{3^{1.5}} \log \log n \le \left(\frac{e}{3^{1.5}}+\frac{1}{\log \log n}\right) \log \log \frac{n}{uv}
$$
if $n$ is large enough, so there is an $N$ such that $W_{u,v} \seq W_{u,v}^*$ for all $(u,v) \in S_n$ whenever $n \geq N$. 
%Indeed, \[\frac{e}{3^{1.5}} \log \log n = \frac{e}{3^{1.5}} \log \log \frac{n}{uv} + \frac{e}{3^{1.5}} \left(\log \log n - \log \log \frac{n}{uv}\right) \tag{$\ast$}\] where the second term tends to $\frac{e}{3^{1.5}}\log 3$ when $n \to \infty$. In particular for $\delta= \frac{1}{2}$, \[
%(\ast) \leq \frac{e}{3^{1.5}} \log \log \frac{n}{uv} + \frac{e}{3^{1.5}} \left(\log 3 + \delta\right)
%\] for sufficiently large $n$. Then \[
%\frac{\log \log \frac{n}{uv}}{\log \log n} = 1- \frac{\log \log n - \log \log \frac{n}{uv}}{\log \log n} \geq 1 - \frac{\log 3 + \delta}{\log \log n}.
%\] Here the right-hand side tends to $1$ as $n \to \infty$, so for large enough $n$, \[
% 1 - \frac{\log 3 + \delta}{\log \log n} \geq \frac{e}{3^{1.5}} \left(\log 3 + \delta\right)
%\] which was what we wanted. 
Hence,
\[
|B_2| \leq \sum\limits_{(u,v) \in S_n} |W_{u,v}| \leq \sum\limits_{(u,v) \in S_n} |W_{u,v}^*|
\] for $n \geq N$.

From Lemma~\ref{ess-l11}, it follows that

\[
|W_{u,v}^*| \leq \frac{\frac{n}{uv}}{\left(\log \frac{n}{uv}\right)^{1+\frac{e}{3^{1.5}} \log \frac{e}{3^{1.5}}- \frac{e}{3^{1.5}} - \frac{\eps}{3}}} = \frac{\frac{n}{uv}}{\left(\log \frac{n}{uv}\right)^{1-\frac{e \log 3}{2\sqrt{3}}- \frac{\eps}{3}}}
\] for $n \geq n_2(\eps)$. Hence,
\begin{equation*}
\begin{split}
    |B_2| &\leq \sum\limits_{(u,v) \in S_n} |W_{u,v}^*| \leq \sum\limits_{(u,v) \in S_n}\frac{\frac{n}{uv}}{\left(\log \frac{n}{uv}\right)^{1-\frac{e \log 3}{2\sqrt{3}}- \frac{\eps}{3}}}\\
    & \ll \frac{n}{(\log n)^{1-\frac{e\log 3}{2 \sqrt{3}}-\frac{\eps}{2}}} \sum\limits_{(u,v) \in S_n} \frac{1}{uv}\\
    & \leq \frac{n}{(\log n)^{1-\frac{e\log 3}{2 \sqrt{3}}-\frac{\eps}{2}}} \left(\sum\limits_{\frac{\sqrt[3]{n}}{(\log n)^{16}} \leq m \leq \sqrt[3]{n}(\log n)^{16}}\frac{1}{m}\right)^2\\
    & \leq \frac{n}{(\log n)^{1-\frac{e\log 3}{2 \sqrt{3}}-\frac{\eps}{2}}} (\log \log n)^2 \ll \frac{n}{(\log n)^{1-\frac{e\log 3}{2 \sqrt{3}}-\eps}}.
\end{split}
\end{equation*}
\end{proof}
    
To prove the upper bounds of Theorems~\ref{thm-k=9-F}, \ref{thm-k=1233-F}, and \ref{thm-k=36-F}, we need one more lemma. The following notation is required: a $d$-equipartition of an integer $n$ is a partition of $n$ into parts of size at least $d$ such that the largest part in the partition is as small as possible. For fixed positive integer $d$ and $n\ge d$, write  $||n||$ for the size of the largest part in a $d$-equipartition of $n$. We will use the following result of Verstra\"{e}te \cite[Section 6.2, Proof of Theorem 2]{Ver}: 

\begin{lemma}\label{Verstraete} 
For every $k\ge d^2$, $d|k$ we have
$$
F_{k,d}(N)-\sum_{i=1}^{\left|\left|\frac{k}{d}\right|\right|-1}\pi \left( \frac{N}{j} \right)\ll N^{1-\frac{1}{2d}}.
$$
\end{lemma}

%\begin{proof}
%We follow Verstra\"{e}te's original proof, %with a little the modification. In the %proof of Theorem 2. in \cite{Ver}  $J=||\frac{k}{d}||$, $n=\lfloor \frac{1}{2}\log _N - \log _2 J \rfloor $ and 
%$$
%Y_i=\{ p \hbox{ prime }: \frac{N}{J2^{i+1}}%<p\le \frac{N}{J2^i}\}    
%$$ for $0\le i\le n$. Verstra\"{e}te used  %the trivial estimation $|Y_i|=O(\frac{N}%{2^i})$. In fact, the bound $|Y_i|%=O(\frac{N}{2^i\log N})$ is also fulfilled, %and this gives the improvement
%$$
%F_{k,d}(N)\le \sum_{j=1}^{||\frac{k}
%{d}||-1}\pi \left( \frac{N}{j} \right)+O\left( \frac{N^{1-\frac{1}{2d}}}{(\log N)^{1-\frac{1}{2d}}} \right).
%$$
%\end{proof}

%\section{Case $k=2$}

\section{Proofs of Theorems}

\begin{proof}[Proof of Theorem~\ref{thm-k=2}]
First, we show that $F_{2,3}(n)=f_{2,3}(n)+1$.

If $A \seq [n]$ is a set such that $a_1a_2=x^3$ has no solution in $A$ with $a_1\neq a_2$, then there is at most one perfect cube in $A$. If we omit this perfect cube from the set, there will be no solution for $a_1a_2=x^3$ in $A$ (where $a_1,a_2$ are not necessarily different): if there was a solution, it would have to be of the form $a_1^2=x^3$, so $a_1$ would have to be yet another perfect cube. Hence $F_{2,3}(n)-1 \leq f_{2,3}(n)$. Vice versa, if $A \seq [n]$ is such that $a_1a_2=x^3$ has no solution in $A$ (with $a_1,a_2$ not necessarily different), then $1 \notin A$, and $a_1a_2=x^3$ has no solution in $A \cup \{1\}$ with $a_1 \neq a_2$. Hence, $f_{2,3}(n)+1 \leq F_{2,3}(n)$.

We prove the theorem for $f_{2,3}(n)$. Let us consider a set $A\seq [n]$ of maximal size avoiding solutions to $a_1a_2=x^3$, by definition, $|A|=f_{2,3}(n)$. We will estimate at least how many elements we have to leave out from $[n]$ to get such a set.

Suppose we have a solution $a_1a_2=x^3$. Let us write $a_1=uv^2w^3$, where $u,v$ are squarefree and $\gcd(u,v)=1$, so $uv^2$ is the cubefree part of $a_1$. Notice that the cubefree part of $a_2$ must be $vu^2$, hence $a_2\in \{vu^2t^3: \ t\in \mathbb{Z}^+ \}$. We will call $a$ and $b$ equivalent if their cubefree parts are the same. We will call $a$ and $b$ {\it opposites} if their product is a perfect cube. Clearly, if a number is in $A$, then there is no number in $A$ from the opposite equivalence class. Therefore, $A$ can be partitioned into equivalence classes, and it contains exactly one (the larger) class from every pair of opposite equivalence classes.

Hence, the minimum number of elements we have to leave out is the sum of the sizes of the smaller equivalence classes (from each opposite pair). The size of the class containing integers with cubefree part $uv^2$ is $\left\lfloor \sqrt[3]{\frac{n}{uv^2}}\right\rfloor$, since for the elements $uv^2t^3\leq n$ the bound $ t \le \sqrt[3]{\frac{n}{uv^2}}$ must hold. For a fixed pair $u,v$  (satisfying $u<v$ and $\gcd(u,v)=1$) we have the following pair of opposite equivalence classes: $\{uv^2t^3: \; t\in \mathbb{Z}^+\} \cap [n]$ and $\{vu^2t^3: \; t\in \mathbb{Z}^+\} \cap [n]$. Since $u<v$, the first class will be smaller (or of the exact same size), so the minimum number of elements we have to leave out is exactly
$$\sum_{\substack{1\leq u<v \\ \gcd(u,v)=1 \\ uv^2\le n \\ u,v \  \text{squarefree}}} \left \lfloor \sqrt[3]{\frac{n}{uv^2}} \right \rfloor =n-f_{2,3}(n).$$

First we give an upper bound for $n-f_{2,3}(n)$ using this formula.

Observe that
$$\sum_{\substack{1\leq u<v, \\ uv \ \text{squarefree}}} \left \lfloor \sqrt[3]{\frac{n}{uv^2}} \right \rfloor = \sum_{\substack{1\leq u<v, \\ uv \ \text{squarefree}}} \sum_{t=1}^{\left \lfloor \sqrt[3]{\frac{n}{uv^2}} \right \rfloor } 1 = \sum_{t=1}^{\left\lfloor \sqrt[3]{n} \right\rfloor}\sum_{\substack{1\leq u<v, \\ uv \: \text{squarefree}, \\ uv^2\le \frac{n}{t^3}, }} 1 < \sum_{t=1}^{\left\lfloor \sqrt[3]{n} \right\rfloor}\sum_{\substack{1\leq u<v,  \\ uv^2\le \frac{n}{t^3} }} 1.$$

The inner sum can be rewritten as
$$\sum_{\substack{1\leq u<v,  \\ uv^2\le \frac{n}{t^3} }} 1=\sum_{u=1}^{\left\lfloor \frac{\sqrt[3]{n}}{t}\right\rfloor}\sum_{v=u+1}^{\left\lfloor \sqrt{\frac{n}{ut^3}}\right\rfloor} 1 \leq \sum_{u=1}^{\left\lfloor \frac{\sqrt[3]{n}}{t}\right\rfloor} \left(\sqrt{\frac{n}{ut^3}}-u\right).$$ 

As $\sqrt{\frac{n}{ut^3}}-u$ is monotonically decreasing in $u$, we get that 
$$\sum_{u=1}^{\left\lfloor \frac{\sqrt[3]{n}}{t}\right\rfloor} \left(\sqrt{\frac{n}{ut^3}}-u\right) < \int\limits_0^{\frac{\sqrt[3]{n}}{t}} \left(\sqrt{\frac{n}{ut^3}}-u\right)\, du = 1.5 \, \frac{n^{2/3}}{t^2}.$$
Hence,
$$\sum_{t=1}^{\left\lfloor \sqrt[3]{n} \right\rfloor}\sum_{\substack{1\leq u<v,  \\ uv^2\le \frac{n}{t^3} }} 1 < \sum_{t=1}^{\left\lfloor \sqrt[3]{n} \right\rfloor} 1.5 \, \frac{n^{2/3}}{t^2}  < 1.5 \, {n^{2/3}} \sum_{t=1}^\infty \frac{1}{t^2}  =  \frac{\pi^2}{4} n^{2/3},$$
which proves the upper bound with $c_2=\frac{\pi^2}{4}$.

Now we provide a lower bound for the number of elements that we have to leave out.

We will construct a set $A_n$ of disjoint pairs $\{y_i,z_i\}$ such that $y_iz_i$ is always a perfect cube. Then we have to leave out at least one element from each pair in $A_n$, hence $|A_n| \leq n-f_{2,3}(n)$.

Let 
\[
    A_n=\left\{\left\{a^2b,ab^2\right\}:\  a,b \in \mathbb{Z}^+ \text{ squarefree},\  a \leq b<\sqrt[3]{n}\right\}.
\] Then $\left(a^2b\right)\left(ab^2\right)=(ab)^3$ is indeed a perfect cube. The squarefree parts of $ab^2$ and $a^2b$ are $a$ and $b$ respectively, which implies that the pairs are disjoint. It is well-known that the number of squarefree numbers less than $m$ is $\left(\frac{6}{\pi ^2}+o(1)\right)m$. Hence
$|A_n| \geq \frac{1}{2}{\left (\frac{6}{\pi ^2} +o(1) \right)^2 n^{2/3}}>0.18n^{2/3}$, which proves the lower bound with $c_1=0.18$.

\end{proof}

%\section{Case $k=3$}

Now, we prove Theorem~\ref{thm-k=3-f} and Theorem~\ref{thm-k=3-F} and give an approximation for the constants $c_{3,3}$ and $C_{3,3}$.

\begin{proof}[Proof of Theorem~\ref{thm-k=3-f}]
Let us take a set $A\seq [n]$ such that the equation $a_1a_2a_3=x^3$ has no solution with $a_1,a_2,a_3\in A$ and integer $x$, except the trivial case $a_1=a_2=a_3$. Note that distinct elements of $A$ have distinct cubefree parts, since $uv^3, uw^3\in A$ with $v\ne w$  would yield the solution $a_1=a_2=uv^3, a_3=uw^3$. Therefore, we can restrict our attention to sets of cubefree numbers, since each element of $A$ may be replaced by its cubefree part. From now on, it is assumed that $A$ contains only cubefree integers.

First, we give an upper bound for the size of $A$. Let $r$ be a fixed positive integer and let $p_i$ denote the $i$th prime. Each cubefree $a\in [n]$ can be written as
$$a=p_1^{\alpha_1}p_2^{\alpha_2}\ldots p_r^{\alpha_r}a',$$
where $\alpha_1,\dots,\alpha_r\in \{0,1,2\}$ and $a'$ is cubefree satisfying $\gcd(a',p_1p_2\dots p_r)=1$. Here $p_1^{\alpha_1}p_2^{\alpha_2}\ldots p_r^{\alpha_r}$ is the \textit{$p_r$-smooth} and $a'$ is the \textit{$p_{r+1}$-rough} part of the number $a$. Observe that the product of three integers is a perfect cube if and only if so are the product of their $p_r$-smooth parts and the product of their $p_{r+1}$-rough parts.
In particular, for a fixed $a'$ there cannot be three elements in $A$ with $p_{r+1}$-rough part $a'$ such that the product of their $p_r$-smooth parts is a perfect cube. Note that the product of three $p_r$-smooth numbers is a cube if and only if the sum of their exponent vectors $(\alpha_1,\alpha_2,\dots,\alpha_r)$ add up to $(0,0,\dots,0)$ calculating coordinate-wise modulo 3. Alternatively, if we consider the exponent vectors as elements of $\mathbb{F}_3^r$, they form a nontrivial 3-term arithmetic progression (3AP). Let $L_r(i)$ be the set of $p_r$-smooth cubefree integers up to $i$:
$$L_r(i):=\left \{p_1^{\alpha_1}p_2^{\alpha_2}\dots p_r^{\alpha_r}:\ \alpha_1,\dots,\alpha_r\in \{0,1,2\} \right \} \cap [i],$$
and let $s_r(i)$ denote the largest possible size of a subset of $L_r(i)$ avoiding nontrivial solutions to $a_1a_2a_3=x^3$. Note that $s_r(i)$ is the size of the largest 3AP-free subset of 
$$\{ (\alpha_1,\dots,\alpha_r)\in \{0,1,2\}^r:\ \alpha_1\log p_1+\dots+\alpha_r \log p_r\leq \log i\},$$
if we consider this set as a subset of $\mathbb{F}_3^r$. For brevity, we will say that $s_r(i)$ is the size of the largest 3AP-free subset of $L_r(i)$. Clearly, for every $i\geq p_1^2\dots p_r^2$, we have $s_r(i)=s_r(p_1^2\dots p_r^2)$ (whose common value is $r_3(\mathbb{F}_3^r)$, the largest possible size of a 3AP-free subset of $\mathbb{F}_3^r$).

For a $p_{r+1}$-rough number $a'$, there are at most $s_r\left(\lfloor n/a' \rfloor\right)$ elements of $A$ that are of the form $a=p_1^{\alpha_1}p_2^{\alpha_2}\ldots p_r^{\alpha_r}a'$. Let us partition $[n]$ into the finite union of intervals $I_i=\left(\frac{n}{i+1},\frac{n}{i} \right]$ (where $1\leq i\leq p_1^2\dots p_r^2-1$) and $I_{p_1^2\dots p_r^2}=\left[1,\frac{n}{p_1^2\dots p_r^2}\right]$. Note that the number of $p_{r+1}$-rough elements of $I_i$ is $|I_i|\cdot \prod\limits_{j=1}^r \left (1-\frac{1}{p_j}\right)+O_r(1)$, where $|I_i|$ denotes the size of $I_i$, since in each set of $p_1p_2\dots p_r$ consecutive integers the proportion of $p_{r+1}$-rough numbers is exactly $\prod\limits_{j=1}^r \left (1-\frac{1}{p_j}\right)$. Therefore,
$$|A|\leq \sum\limits_{i=1}^{p_1^2\dots p_r^2} |I_i|\cdot \prod\limits_{j=1}^r \left (1-\frac{1}{p_j}\right)s_r(i)+O_r(1).$$
Setting 
$$\gamma_r:=\left(\frac{s_r(p_1^2\dots p_r^2)}{p_1^2\dots p_r^2}+\sum\limits_{i=1}^{p_1^2\dots p_r^2-1} \frac{s_r(i)}{i(i+1)}\right)\prod\limits_{j=1}^r \left (1-\frac{1}{p_j}\right),$$ 
it is obtained that $|A|\leq \gamma_rn+O_r(1)$, consequently, 
\begin{equation}\label{eq-upp}
    \limsup\limits_{n\to\infty} \frac{f_{3,3}(n)}{n}\leq \inf \{\gamma_r:\ r\geq 1\}.
\end{equation}
%The sequence $\gamma_r$ is in fact decreasing, but we will not make use of this, instead, as a next step, we show a matching lower bound.

In the previous upper bound we used only that there is no nontrivial solution to $a_1a_2a_3=x^3$, where the elements $a_1,a_2,a_3\in A$ have the {\it same} $p_{r+1}$-rough part. However, if we consider only squarefree numbers as possible $p_{r+1}$-rough parts $a'$, then we will get a suitable set, as among the squarefree numbers the equation $a_1a_2a_3=x^3$ has only  trivial solutions, where $a_1=a_2=a_3$.
Indeed, in the prime factorization of a product of three squarefree numbers each exponent is 1, 2, or 3, so in order to get a perfect cube, the three numbers must have exactly the same prime factors, meaning they coincide.  Let us define $A$ in the following way: for each $p_{r+1}$-rough squarefree $a'\leq n$ choose $s_r(\lfloor n/a' \rfloor )$ elements $t_1, t_2, \dots, t_{s_r(\lfloor n/a' \rfloor )}$ of $L_r(\lfloor n/a' \rfloor)$ in such a way that they form a 3AP-free subset of $L_r(\lfloor n/a' \rfloor)$, and put $a't_1,\dots,a't_{s_r(\lfloor n/a' \rfloor )}$ into $A$. 

Now, assume that $a_1a_2a_3$ is a perfect cube for some $a_1,a_2,a_3\in A$. Then the product of the $p_{r+1}$-rough parts of $a_1,a_2,a_3$ is also a perfect cube, but they are all squarefree, so the $p_{r+1}$-rough parts of $a_1,a_2,a_3$ must be the same, say $a'$. The product of the $p_r$-smooth parts is also a perfect cube, but multiples of $a'$ were added to $A$ in such a way that a three-factor product is a cube only if the $p_{r+1}$-smooth parts are also the same, thus $a_1=a_2=a_3$. Hence, $A$ does in fact satisfy the required property.

To estimate the size of $A$ we need bounds for the number of squarefree $p_{r+1}$-rough numbers in the intervals $I_i$. The $p_{r+1}$-rough numbers can be partitioned into residue classes modulo $p_1\dots p_r$. By an old result of Prachar~\cite{Prac} we get that
\begin{multline*}|\{a'\in I_i:\ \gcd(a',p_1\dots p_r)=1\text { and $a'$ is squarefree}\}|=\\
=|I_i|\cdot \prod\limits_{j=1}^r \left(1-\frac{1}{p_j}\right)\prod\limits_{j>r} \left(1-\frac{1}{p_j^2}\right)+O_r(\sqrt{n}).
\end{multline*}
Therefore,
$$|A|\geq \sum\limits_{i=1}^{p_1^2\dots p_r^2} |I_i|\cdot \prod\limits_{j=1}^r \left (1-\frac{1}{p_j}\right)\prod\limits_{j>r} \left (1-\frac{1}{p_j^2}\right)s_r(i)-O_r(\sqrt{n}).$$
Setting 
\begin{equation*}%\label{eq-low}
\beta_r:=\gamma_r \cdot \prod\limits_{j>r} \left (1-\frac{1}{p_j^2}\right),
\end{equation*}
it is obtained that $|A|\geq \beta_rn-O_r(\sqrt{n})$, consequently, 
\begin{equation}\label{eq-low}
\liminf\limits_{n\to\infty} \frac{f_{3,3}(n)}{n}\geq \sup \{\beta_r:\ r\geq 1\}.
\end{equation}
Since $\gamma_r/\beta_r\to 1$ (as $r\to \infty$), by comparing \eqref{eq-upp} and \eqref{eq-low} we obtain that $f_{3,3}(n)/n$ converges to $\inf \{\gamma_r:\ r\geq 1\}=  \sup \{\beta_r:\ r\geq 1\}=:c_{3,3}$.

\end{proof}

Numerically we obtained the bounds $0.6224\leq c_{3,3} \leq 0.6420$ by taking $r=4$ and calculating by computer the values $s_4(i)$, which are given in Appendix \ref{table1}.

\begin{proof}[Proof of Theorem~\ref{thm-k=3-F}]

A slight modification of the proof of Theorem~\ref{thm-k=3-f} gives the analogous result for the function $F_{3,3}$; here we only highlight the differences.  
Note that here two elements of $A$ might have the same cubefree part, since solutions like $a_1=a_2=uv^3, a_3=uw^3$ are excluded here. However, {\it three} elements of $A$ can not have the same cubefree part, since $uv^3, uw^3, uz^3\in A$ with distinct $v,w,z$ would give the solution $a_1=uv^3,a_2=uw^3,a_3=uz^3$.

Therefore, we may assume that each element of $A$ is either cubefree, or a cubefree number multiplied by $8$. This leads to the following modification of the definition of $s_r(i)$: let $S_r(i)$ be the largest possible total weight of a 3AP-free subset of $L_r(i)$, where $s\in L_r(i)$ weighs $1$ if $s>\frac{i}{8}$, and weighs $2$ if $s\leq \frac{i}{8}$, since in this case $8s$ can also be chosen into ${A}$ beside $s$.

%Hence, in this case instead of $s_r(i)$, we want to find the 3-AP subset of $L_r(i)$ with the biggest weight, where $q\in L_r(i)$ weights 1 if $q>\frac{i}{8}$, and weights 2 if $q\leq \frac{i}{8}$, since in this case $8q$ can also be chosen into ${A}$. Let us denote the weight of this subset with $S_r(i)$. 

Clearly, for every $i\geq 8p_1^2\dots p_r^2$, we have $S_r(i)=S_r(p_1^2\dots p_r^2)$.
Now, partition $[n]$ into the finite union of intervals $I_i=\left(\frac{n}{i+1},\frac{n}{i} \right]$ (where $1\leq i\leq 8p_1^2\dots p_r^2-1$) and $I_{8p_1^2\dots p_r^2}=\left[1,\frac{n}{8p_1^2\dots p_r^2}\right]$. 
%Since everything that we claimed earlier about the $p_{r+1}$-rough elements of the intervals is also true here, 
Analogously to the proof of Theorem~\ref{thm-k=3-f} we get that
$$|A|\leq \sum\limits_{i=1}^{8p_1^2\dots p_r^2} |I_i|\cdot \prod\limits_{j=1}^r \left (1-\frac{1}{p_j}\right)S_r(i)+O_r(1).$$ 
Setting 
$$\Gamma_r:=\left(\frac{S_r(8p_1^2\dots p_r^2)}{8p_1^2\dots p_r^2}+\sum\limits_{i=1}^{8p_1^2\dots p_r^2-1} \frac{S_r(i)}{i(i+1)}\right)\prod\limits_{j=1}^r \left (1-\frac{1}{p_j}\right),$$ 
it is obtained that $|A|\leq \Gamma_rn+O_r(1)$, consequently, 
\begin{equation}\label{eqF-upp}
    \limsup\limits_{n\to\infty} \frac{F_{3,3}(n)}{n}\leq \inf \{\Gamma_r:\ r\geq 1\}.
\end{equation}

%Now we give a lower bound similar to the previous case, but with the additional weights.
%Using the same notations as before, the only thing that changes here is that we consider $S_r(i)$ instead of $s_r(i)$, and more intervals.
%Therefore,
As a lower bound, we get that
$$|A|\geq \sum\limits_{i=1}^{8p_1^2\dots p_r^2} |I_i|\cdot \prod\limits_{j=1}^r \left (1-\frac{1}{p_j}\right)\prod\limits_{j>r} \left (1-\frac{1}{p_j^2}\right)S_r(i)-O_r(\sqrt{n}).$$
Setting 
\begin{equation*}%\label{eqF-low}
B_r:=\Gamma_r \cdot \prod\limits_{j>r} \left (1-\frac{1}{p_j^2}\right),
\end{equation*}
we obtain that $|A|\geq B_rn-O_r(\sqrt{n})$, consequently, 
\begin{equation*}\label{eqF-low}
\liminf\limits_{n\to\infty} \frac{F_{3,3}(n)}{n}\geq \sup \{{B}_r:\ r\geq 1\}.
\end{equation*}
Since $\Gamma_r/B_r\to 1$ (as $r\to \infty$), by comparing \eqref{eqF-upp} and \eqref{eqF-low} we get that $F_{3,3}(n)/n$ converges to the limit $\inf \{\Gamma_r:\ r\geq 1\}=  \sup \{B_r:\ r\geq 1\}=:C_{3,3}$.

\end{proof}

We approximated this constant as well, relying on the previous approximation. Note that when $a' > \frac{n}{8}$, nothing changes, but this is not necessarily true for $a' \leq \frac{n}{8}$. In the previous case, for example, we could choose $6$ elements when $\frac{n}{10}\leq a' < \frac{n}{8}$, including $a'$. Now we can also include $8a'$ in the set, so the maximum weight will be $7$. We checked every interval by computer to find the subset with the maximum weight and obtained the estimates $0.6919\leq C_{3,3} \leq 0.7136$. (For the values of $S_4(i)$ we calculated see the table in Appendix~\ref{table2}.)

% EZ A C_3,3 TÁBLÁZAT NAGY C-RE ÉS R=3-RA

%\renewcommand{\arraystretch}{1.5}
%\begin{center}
%\begin{tabular}{ | c |c| } 
% \hline
% The interval of $a'$ & $S_3(\lfloor n/a'\rfloor)$ \\
% \hline\hline
% $a' \in (n/2, n]$ & 1 \\ 
% \hline 
% $a' \in (n/3, n/2]$ & 2 \\ 
% \hline
% $a' \in (n/5, n/3]$ & 3 \\ 
 %\hline
 %$a' \in (n/6, n/5]$ & 4 \\ 
 %\hline
 %$a' \in (n/8, n/6]$ & 5  \\ 
 %\hline
 %$a' \in (n/10, n/8]$ & 6  \\ 
 %\hline
 %$a' \in (n/15, n/10]$ & 7 \\ 
 %\hline
 %$a' \in (n/16, n/15]$ & 8  \\
 %\hline
 %$a' \in (n/24, n/16]$ & 9  \\
 %\hline
 %\end{tabular}
 %\quad
 %\begin{tabular}{ | c |c| } 
 %\hline 
 % The interval of $a'$ & $S_3(\lfloor n/a' \rfloor)$ \\
 %\hline\hline
 %$a' \in (n/30, n/24]$ & 10  \\ 
 %\hline
 %$a' \in (n/40, n/30]$ & 11  \\
 % \hline
 %$a' \in (n/48, n/40]$ & 12 \\
 %\hline
 % $a' \in (n/60, n/48]$ & 13  \\
 %\hline
 % $a' \in (n/80, n/60]$ & 14  \\
 %\hline
 % $a' \in (n/120, n/80]$ & 15  \\
 %\hline
 % $a' \in (n/200, n/120]$ & 16  \\
 %\hline
 % $a' \in (n/240, n/200]$ & 17  \\
 %\hline
 %$a' \leq n/240$ & 18  \\ 
 %\hline
%\end{tabular}
%\bigskip
%\end{center}
%\renewcommand{\arraystretch}{1}

\noindent

\begin{proof}[Proof of Theorem~\ref{thm-k=4}]
First we prove the lower bound.

Let $A \seq [n]$ be a subset  such that $a_1a_2a_3a_4 \neq x^3$ if $a_i \in A$, $a_1<a_2<a_3<a_4$ and let $D = \{d_1, \ldots, d_t\}$ be the set of all positive integers $d$ such that $d \leq n^{1/3}$ and $\Omega(d) \leq \frac{1}{3} \log \log n$. Then by Lemma~\ref{ess-l11},
\begin{equation*}
    \begin{split}
        t=|D| &\geq |\{d:\ d \leq n^{1/3},\ \Omega(d) \leq \frac{1}{3} \log \log n^{{1}/{3}}\}|>\\
        & > \frac{n^{1/3}}{\left(\frac{1}{3} \log n\right)^{1+\frac{1}{3} \log \frac{1}{3} - \frac{1}{3} + \frac{\eps}{3}}} > \frac{n^{1/3}}{(\log n)^{1+\frac{1}{3} \log \frac{1}{3} - \frac{1}{3} + \frac{\eps}{3}}},
    \end{split}
\end{equation*}
if $n \geq n_2\left(\eps\right)$.

Let $H$ be the $3$-uniform hypergraph on the vertex set $\{P_1, \ldots, P_t\}$ such that $\{P_i, P_j, P_k\}$ is an edge in $H$ if and only if $d_id_jd_k \in A$. Let $M$ be the set of those $m \in [n]$ such that $m \notin A$ and $m=d_id_jd_k$ for some $1 \leq i < j < k \leq t$, then $|A| \leq n - |M|$.

For a fixed $m \in M$ let $h(m)$ denote the number of triples $(d_i, d_j, d_k)$ such that $m=d_id_jd_k$, $1 \leq i < j < k \leq t$. If $m=p_1^{k_1}p_2^{k_2}\cdots p_r^{k_r} \in M$, then \[
\Omega(m)=\Omega(d_i)+\Omega(d_j)+\Omega(d_k) \leq \log \log n,
\] hence \[
h(m) \leq \tau_3(m) = \prod\limits_{i=1}^r \binom{k_i+2}{2} \leq \prod\limits_{i=1}^r 3^{k_i} = 3^{\Omega(m)} \leq 3^{\log \log n} = (\log n)^{\log 3},
\] 
where $\tau_3(m)$ denotes the number of triples $(a,b,c)$ with $a,b,c\in\zz^+$ such that $m=abc$. 

If $H$ contains a $K_4^3$ (a subhypergraph $G$ with vertex set $V=\{P_{i_1}, P_{i_2}, P_{i_3}, P_{i_4}\}$ such that $V \setminus \{P_{i_j}\}$ is an edge in $G$ for every $j \in \{1,2,3,4\}$), then for some $d_{i_1}<d_{i_2}<d_{i_3}<d_{i_4}$ and $$a_1=d_{i_1}d_{i_2}d_{i_3},\quad  a_2=d_{i_1}d_{i_2}d_{i_4}, \quad a_3=d_{i_1}d_{i_3}d_{i_4},  \quad a_4=d_{i_2}d_{i_3}d_{i_4}$$
we have $a_1<a_2<a_3<a_4$, $a_1,a_2,a_3,a_4\in A$ and $a_1a_2a_3a_4=(d_{i_1}d_{i_2}d_{i_3}d_{i_4})^3$. Therefore, $H$ does not contain any $K_4^3$. Therefore, by Lemma~\ref{Caen} there exists a constant $\delta > 0$ such that there at least $\delta t^3$ triples $(i,j,k)$, $1 \leq i < j < k \leq t$ such that $\{P_i, P_j, P_k\}$ is not an edge in $H$.% (see \cite{Chu}).

Let $h=\max\limits_{m \in M} h(m) \leq (\log n)^{\log 3}$. If $\{P_i,P_j,P_k\} \notin H$, $1 \leq i < j < k \leq t$, then $m=d_id_jd_k$ has at most $h$ decompositions as a product of three positive integers, which gives the following bound on $M$: \[
|M| \geq \frac{\delta t^3}{h} \gg \frac{n}{(\log n)^{3+\log \frac{1}{3} -1 + \eps}\cdot (\log n)^{\log 3}} = \frac{n}{(\log n)^{2+\eps}},
\] which completes the proof of the lower bound.

Now, we prove the upper bound. Let $A_n$ denote the set of the integers $a$ such that
\begin{enumerate}[(i)]
    \item $\frac{n}{\log n} \leq a \leq n$,
    \item $d^2 \mid a$ implies $d \leq \log n$, and
    \item $a$ cannot be written in the form $a=uvw$ with integers $u,v,w$ such that $\frac{\sqrt[3]{n}}{(\log n)^{16}}\le u,v,w\le \sqrt[3]{n}(\log n)^{16}$.
\end{enumerate}

It follows directly from Lemma~\ref{k=4-l1} that the equation $$a_1a_2a_3a_4=x^3$$ has no solution in $A_n$. Hence $f_{4,3}(n) \geq |A_n|$. On the other hand, let $\eps > 0$ and $n \geq n_0(\eps)$ as per Lemma~\ref{k=4-l2}. Note that for $n \geq 3$, the number of integers satisfying $(i)$ and $(ii)$ in the definition of $A_n$ is at least $n-\frac{3n}{\log n}$ (see \cite[Proof of Theorem 2]{ESS}).
By Lemma~\ref{k=4-l2}, with the exception of at most $\frac{n}{(\log n)^{1-\frac{e\log 3}{2\sqrt{3}}-\frac{\eps}{2}}}$ of these integers also satisfy $(iii)$, so we have \[
|A_n| > n-\frac{3n}{\log n} - \frac{n}{(\log n)^{1-\frac{e\log 3}{2\sqrt{3}}-\frac{\eps}{2}}} > n - \frac{n}{(\log n)^{1-\frac{e\log 3}{2\sqrt{3}}-\eps}},
\] which completes the proof.
\end{proof}

%\section{Case $d\nmid k$}

\begin{proof}[Proof of Theorem~\ref{lowerbound}.] Let $\frac{1}{3}\le \alpha \le \frac{1}{2}$. Let $A_{\alpha ,n}$ denote the set of those positive integers $a$ such that $a\le n$ and $a$ has exactly one  prime divisor $p$ greater than $n^{\alpha }$ such that $p^2 \nmid a$.

In the following, integers $p$ and $q$ denote prime numbers. Then, by
$$
\sum_{p\le x}\frac{1}{p}=\log \log x +M +o(1),
$$
(where $M$ is the Meissel-Mertens constant) and the Prime Number Theorem we have
\begin{equation*}
\begin{split}
\left|A_{\alpha ,n}\right|& = \sum_{n^{\alpha }<p\le n}\left\lfloor  \frac{n}{p} \right\rfloor - \sum_{\substack {(p,q)\\ n^{\alpha }<p,q}} \left\lfloor \frac{n}{pq} \right\rfloor =n\left(  \sum_{n^{\alpha }<p\le n}  \frac{1}{p} - \sum_{n^{\alpha }<p\le n^{1-\alpha }}\frac{1}{p}\sum_{n^{\alpha}<q\le \frac{n}{p}} \frac{1}{q} +o(1) \right)\\ 
& = n\left( -\log \alpha - \sum_{n^{\alpha }<p\le n^{1-\alpha }}\frac{\log \log \frac{n}{p}-\log \log n^{\alpha }}{p} +o(1)\right)\\
&=   n\left( -\log \alpha + \sum_{n^{\alpha }<p\le n^{1-\alpha }}\frac{\log \alpha -\log \left( 1-\frac{\log p}{\log n} \right)}{p} +o(1)\right)\\
& =n( -\log \alpha +\left( \log (1-\alpha )-\log \alpha  \right)\log \alpha- \\
&-\sum_{n^{\alpha +o(1)}\le k\le n^{1-\alpha +o(1)}}\frac{\log \left( 1-\frac{\log (k\log k)}{\log n} \right)}{k\log k}  +o(1) )\\
&=n\left( -\log \alpha +\left( \log (1-\alpha )-\log \alpha  \right)\log \alpha -\sum_{n^{\alpha }\le k\le n^{1-\alpha }}\frac{\log \left( 1-\frac{\log k}{\log n} \right)}{k\log k}  +o(1) \right) \\
&=n\left( -\log \alpha +\left( \log (1-\alpha )-\log \alpha  \right)\log \alpha -\int_{n^{\alpha }}^{n^{1-\alpha }}\frac{\log \left( 1-\frac{\log x}{\log n} \right)}{x\log x}dx  +o(1) \right) \\
&=n\left( -\log \alpha +\left( \log (1-\alpha )-\log \alpha  \right)\log \alpha -\int_{\alpha }^{1-\alpha }\frac{\log (1-t)}{t}dt  +o(1) \right),
\end{split}
\end{equation*}
as $n\to \infty $. 

Moreover, if $a_1,\dots ,a_k\in A_{\alpha ,n}$, then each  $a_i$ has exactly one prime divisor $p$ greater than $n^{\alpha}$, moreover, $p^2 \nmid a$. Therefore, the number of those prime divisors of $a_1\dots a_k$ that are larger than $n^\alpha$ is exactly $k$ (counted by multiplicity). As $d\nmid k$, this implies that $a_1\dots a_k$ cannot be a perfect $d$-th power, so $A\in \gamma _{k,d}$.

%It follows that there is a prime $q>n^{\alpha}$ such that defining $r=r(q)$ by $q^r\mid a_1\dots a_k$, $q^{r+1}\nmid a_1\dots a_k$, $d\nmid r$. Thus $a_1\dots a_k$ cannot be a $d$-power, so $A\in \Gamma _{k,d}$. 

\end{proof}

\begin{proof}[Proof of Theorem~\ref{upperbound}] 
Let us take a set $A\subseteq [n]$, $A\in \Gamma _{k,d}$. For brevity, write $t=\pi (n^{1/d})=(d+o(1))\frac{n^{1/d}}{\log n}$. Let $B$ denote the set of integers $b\in A$ such that $b$ is the product of $d$ distinct primes not larger than $n^{1/d}$, that is, $b=p_{i_1}\dots p_{i_d}$ with $1\le i_1<i_2<\dots <i_d\le t $.

Let us define the $d$-uniform hypergraph $H(B)$ on the vertex set $\{P_1,\dots ,P_t\}$ so that $P_{i_1},\dots ,P_{i_d}$ form an edge if and only if $p_{i_1}\dots p_{i_d}\in B$. Since $B\subseteq A\in \Gamma_{k,d}$, the hypergraph $H(B)$ cannot contain a $K^d_k$, otherwise there would exist prime numbers $p_{i_1}<\dots < p_{i_k}$ such that for $$a_1=p_{i_1}\dots p_{i_d},a_2=p_{i_2}\dots p_{i_{d+1}},\dots ,a_k=p_{i_k}p_{i_1}\dots p_{i_{d-1}}$$
we have $a_i\in A$ and
$$
a_1a_2\dots a_k=(p_{i_1}p_{i_2}\dots p_{i_k})^d,
$$
which would contradict that $A\in \Gamma_{k,d}$. 

By Lemma~\ref{Caen}, the number of $d$-tuples $(P_{i_1},\dots ,P_{i_d})$ which do not form an edge in $H(B)$  is at least
$$
\binom{t}{d}-\left ( 1-\frac{1}{\binom{k-1}{d-1}} +\varepsilon \right) \binom{t}{d}\gg _{k,d} \frac{n}{(\log n)^d}.
$$

Hence, at least $c_{k,d}\frac{n}{(\log n)^d}$ integers of the form $p_{i_1}\dots p_{i_d}\le n$ are missing from $A$, as we claimed.
\end{proof}

%\section{Case $k=6$}

\begin{proof}[Proof of Theorem~\ref{thm-k=6-f}]
The lower bound is a consequence of Lemma~\ref{exgraph} (1), Lemma~\ref{k=6-l1} and Prime Number Theorem.

%First, we prove the lower bound. Let ${P}=\{ p_1,p_2,\dots ,p_t \}$ denote the set of primes not exceeding $\sqrt{n}$. By the Prime Number Theorem,
%$$
%t=(2+o(1))\frac{\sqrt{n}}{\log n}.
%$$
%By Lemma~\ref{exgraph} (1) there exists a $\{C_3,C_4\}$-free graph $G_t$ on the vertex set $\{P_1,\dots ,P_t\}$ such that 
%$$
%E(G_t)\gg t^{3/2}.
%$$
% Let ${B}$ be the set of integers of the form $p_ip_j$ where $1\le i<j\le t$ and $P_i$ and $P_j$ are joined in $G_t$, thus $|{B}|=E(G_t)$. Then, by Lemma~\ref{k=6-l1} we have ${B}\in \gamma_{6,3}$.  Let ${C}$ denote the set of primes from the interval $(\sqrt{n}, n]$, then $
%|{C}|=\pi (n)-t$.
%Let us consider the set ${A}={B}\cup {C}$,
%then 
%$$
%|{A}|-\pi(n)\gg \frac{n^{3/4}}{(\log n)^{3/2}}.
%$$
%Now we prove that ${A}\in \gamma_{6,3}$. Assume that  $a_1a_2\dots %a_6=y^3 $ for some $a_1,\dots,a_6\in A$ and $y\in \mathbb{Z}^+$. As $B\in  \gamma_{6,3}$, we may assume that, say, $a_1$ is in $C$. Then the multiset $\{a_1,a_2,\dots,a_6\}$ must contain (at least) three copies of $a_1$, say, $a_1=a_2=a_3$, which also implies that $a_4=a_5=a_6$, giving a trivial solution. 

%If $a_6\in {C}$, then $a_4=a_5=a_6$ and therefore $a_1=a_2=a_3$, which would give a trivial solution. If $a_6\notin {C}$, then $a_1,a_2,\dots ,a_6\in {B}$, again their product cannot be a cube. 

To prove the upper bound, assume that ${A}\subseteq [n]$ and
$$
|{A}|\ge \pi(n)+c\frac{n^{3/4}}{(\log n)^{3/2}},
$$
where $c$ is so large that $A$ cannot be a multiplicative Sidon set. By  \cite[Lemma 14]{ESS} there are four distinct integers $a_1,a_2,a_3,a_4\in {A}$ such that $a_1a_2=a_3a_4$, so $a_1^2a_2^2a_3a_4=(a_1a_2)^3$ and thus ${A}\notin \gamma_{6,3}$.  
\end{proof}

\begin{proof}[Proof of Theorem~\ref{thm-k=6-F}]
First we prove the lower bound. We build on ideas from \cite{Erdos_1964}. Let ${P}$ be the set of prime numbers. Let us consider the set
$$
A=\left\{ m:\ m=pq,\ \frac{n}{\log n}<m\le n,\ p,q\in P,\ p<\frac{q}{\log n}\right\} .
$$
For the size of $A$ we give the following lower bound:
\begin{equation*}
\begin{split}
|A|&=\pi_2(n)-\pi_2\left(\frac{n}{\log n}\right)-\left|\left\{ m:\ m=pq,\ \frac{n}{\log n}<m\le n,\ p,q\in P,\ \frac{q}{\log n} \le p\leq q\right\}\right|\ge \\
& \geq (1-o(1))\frac{n\log \log n}{\log n}-\left|\left\{ m:\ m=pq,\ p\le q\le \sqrt{n\log n} \right\}\right|=(1-o(1))\frac{n\log \log n}{\log n}.
\end{split}
\end{equation*}

Now we show that $A\in \Gamma_{6,3}$. Let us assume that for some distinct elements $a_1,\dots,a_6\in A$  we have $a_1a_2\dots a_6=z^3$ for some $z\in \mathbb{Z}^+$. Then there exist prime numbers $p_1<p_2<p_3<p_4$ such that 
$$
\{ a_1,a_2,a_3,a_4,a_5,a_6\} =\{ p_1p_2, p_1p_3,p_1p_4, p_2p_3,p_2p_4,p_3p_4\} .
$$
Since $\frac{n}{\log n}< p_3p_4\le n$ and $p_3<\frac{p_4}{\log n}$, we get that $p_3\le \frac{\sqrt{n}}{\sqrt{\log n}}$. Thus $p_1p_2<\frac{n}{\log n}$, which contradicts the definition of $A$. Hence, $A\in \Gamma_{6,3}$.

To prove the upper bound, let us suppose that $\delta >0$ and for $A\subseteq [n]$ we have $|A|\ge (1+\delta )\frac{n\log \log n}{\log n}$. By \cite[Theorem 3]{Erdos_1964}, if $n$ is large enough, there exist distinct $a_1,a_2,\dots a_6\in A$ such that
$$
a_1a_2=a_3a_4=a_5a_6.
$$
Therefore, $a_1a_2a_3a_4a_5a_6$ is a cube, thus $A\notin \Gamma_{6,3}$.

\end{proof}

%\section{Case $k=9$}

\begin{proof}[Proof of Theorem~\ref{thm-k=9-f}]

The lower bound is a consequence of Lemma~\ref{exgraph} (2), Lemma~\ref{k=6-l1} and Prime Number Theorem.

Now we prove the upper bound.

%%%%%%%%%%%
%%%%%%%%%%%

Let $A\subseteq [n]$ be a set such that $a_1\dots a_9=x^3$ has no solution in $A$, except trivial solutions of the form $a^3b^3c^3=x^3$. First, we show that the equation $a_1a_2a_3=b_1b_2b_3$ has no solution in $A$ with $a_1\leq a_2\leq a_3,\ b_1\leq b_2\leq b_3$ except the trivial solutions where $(a_1,a_2,a_3)=(b_1,b_2,b_3)$. For the sake of contradiction, assume that a nontrivial solution exists. Observe that we may assume that one of the $a_i$ or $b_i$ appears only once in the multiset $\{a_1,a_2,a_3,b_1,b_2,b_3\}$. Indeed, if $a_1=a_2=a_3$, then $b_1=b_2=b_3$ would give a trivial solution, so one of the $b_i$ appears only once (and it has to be different from $a_1=a_2=a_3$). 
If $a_1=a_2<a_3$, then $a_3$  or $b_3$ appears only once, unless $a_3=b_3$. If $a_3=b_3$, then $b_1=b_2$ gives a trivial solution and $b_1<b_2$ yields that $b_1$  appears only once. The case $a_1<a_2=a_3$ can be handled in a similar way. Finally, if $a_1<a_2<a_3$, then one of the $a_i$ appears only once. 

So we may assume that $a_t$ appears only once in the multiset $\{a_1,a_2,a_3,b_1,b_2,b_3\}$. Then $a_1^2a_2^2a_3^2b_1b_2b_3=x^3$ is a nontrivial solution, since the multiplicity of $a_t$ is exactly 2.

Therefore, it suffices to show that for any set $A\subseteq [n]$ avoiding nontrivial solutions to $a_1a_2a_3=b_1b_2b_3$ ($a_1\leq a_2\leq a_3$, $b_1\leq b_2\leq b_3$) we have $|A|\leq \pi(n)+n^{2/3}$.

Note that $A$ must be a multiplicative 3-Sidon set, for the maximal possible size of these the best upper bound is $(\pi (n)+\pi (\frac{n}{2}))\ll n^{2/3}(\log n)^{2^{1/3}-1/3+o(1)}$, and the main term $\pi(n)+\pi(n/2)$ is tight. However, in our case solutions like $pq(2r)=p(2q)r$ are also excluded, consequently we shall adopt the proof of this bound to our setting. 

Let us express every element $a\in A$ in the form $a=uv$, where $v\leq u$ and either $u\leq n^{2/3}$ or $u$ is a prime. Let us consider a graph $G$ on vertex set $[n]$ where $uv$ is an edge if and only $uv$ is the representation of an element $a\in A$. Let us first consider the subgraph $G_1$ formed by the edges where $u>n^{2/3}$ is a prime. Observe that $G_1$ is $C_4$-free. Indeed, if $pa,qa,pb,qb$ form a $C_4$ (where $n^{2/3}<p,q$ are primes), then $(pa)(qb)c=(pb)(qa)c$ would be a nontrivial solution according to our definition. Since $G_1$ is $C_4$-free, the number of edges in $G_1$ is at most $\pi(n)-\pi(n^{2/3})+\Theta(n^{2/3})$ by \cite[Lemma 2]{Erdős}. For the number of remaining edges we get the bound $O(n^{2/3}\log n)$ in the same way as in the proof for the multiplicative 3-Sidon bound. Hence, the number of edges in $G$ is at most $\pi(n)+O(n^{2/3}\log n)$, completing the proof of the theorem.

\end{proof}

\noindent

\begin{proof}[Proof of Theorem~\ref{thm-k=9-F}]

The upper bound is a trivial consequence of Lemma~\ref{Verstraete} with $k=9$, $d=3$.

For the lower bound, we give a construction similar to the one in the proof of {Theorem~\ref{thm-k=9-f}}.

Let $A_0=\{p:\; \sqrt{n}<p \leq n,  \;p \text{ prime}\}\cup \{2p:\; \sqrt{n}<p \leq n/2,\; p\text{ prime}\}$. Let us again partition the primes not exceeding $\sqrt{n}$ into two sets $S$ and $T$ such that $|S|=|T|$, leaving out one prime if $\pi(\sqrt{n})$ is odd, then $$|S|=|T|=\Theta\left(\pi\left(\sqrt{n}\right)\right)=\Theta\left(\frac{\sqrt{n}}{\log n}\right).$$
Lemma~\ref{k33} guarantees the existence of a $K_{3,3}$-free bipartite graph $G=(S,T,E)$ with at least $c|S|^{5/3}$ edges for some constant $c>0$. Let $st \in A_1$ if and only if $(s,t) \in E$, and define $A=A_0 \cup A_1$. Clearly, 
\begin{equation*}
    \begin{split}
        |A|=|A_0|+|A_1| &\geq \left(\pi(n) - \pi\left(\sqrt{n}\right)\right) + \left(\pi\left(\frac{n}{2}\right) - \pi\left(\sqrt{n}\right)\right) + c|S|^{5/3}=\\
        &=\pi(n) + \pi\left(\frac{n}{2}\right) + \Theta\left(\frac{n^{5/6}}{(\log n)^{5/3}}\right).
    \end{split}
\end{equation*}

Now we show that the equation 
$$a_1\ldots a_9=x^3,\quad a_1<\ldots<a_9$$
has no solution in $A$ (that is, $A \in \Gamma_{9,3}$).

It is clear that a potential solution $a_1, \ldots, a_9 \in A$ cannot contain any elements from $A_0$. Suppose that the equation has a solution $a_1, \ldots, a_9$ in $A_1$. Then $a_1, \ldots, a_9$ are represented by $9$ edges in $G$. In order for $a_1 \ldots a_9$ to be a perfect cube, these edges must form a subgraph $H$ such that every vertex has degree divisible by $3$ in $H$. This implies that $H$ must contain at least $3$-$3$ vertices from both $S$ and $T$. However, since $G$ is $K_{3,3}$-free by construction, $H \neq K_{3,3}$, so $H$ contains at least $4$ vertices from at least one of the vertex sets $S$ and $T$. This yields that $a_1\ldots a_9$ has at least $7$ distinct prime divisors, thus $\Omega\left(x^3\right) \geq 21$. On the other hand, however, $\Omega\left(x^3\right)=\sum\limits_{i=1}^9\Omega(a_i)=18$. Hence, we have reached a contradiction, which concludes the proof.

\end{proof}

\begin{proof}[Proof of Theorem~\ref{thm-k=1233-F}] The proofs of the last five inequalities are very similar to the proofs of the first three, so only the first three are proved.

(1) First we prove the lower bound.  Let
\begin{equation*}
    \begin{split}
        B= & \{ p:\sqrt{n}<p\le n,\  p \hbox{ prime}\} \cup \{ 2p:\sqrt{n}<p\le \frac{n}{2},\ p \hbox{ prime}\}\cup\\
        &\cup \{ 3p:\sqrt{n}<p\le \frac{n}{3},\ p \hbox{ prime}\}.
    \end{split}
\end{equation*}

By the Prime Number Theorem,
$$
|B|=\pi (n)+\pi \left(\frac{n}{2}\right)+\pi \left(\frac{n}{3}\right)-O\left(\frac{\sqrt{n}}{\log n}\right).
$$
Let $P=\{ p_1,p_2,\dots ,p_t\} $ denote the set of primes $p$ with $3<p\le \sqrt{n}$, by the Prime Number Theorem $t=(2+o(2))\frac{\sqrt{n}}{\log n}$. Let $G$ be a $\{C_3,C_{4}\} $-free graph on vertex set $\{P_1,\dots ,P_t\}$ with the maximal possible number of edges. Let $C$ denote the set of elements of the form $p_ip_j$ (where $1\le i<j\le t$) such that $p_ip_j\in C$ if and only if $P_iP_j$ is an edge in $G$. By Lemma~\ref{exgraph} (1) we have 
$$
|C|\gg \frac{n^{3/4}}{(\log n)^{3/2}}.
$$
Thus, for the size of $A=B\cup C$ we have the lower bound
$$
|A|=|B|+|C|\ge \pi (n)+\pi \left(\frac{n}{2}\right)+\pi \left(\frac{n}{3}\right)+c\frac{n^{3/4}}{(\log n)^{3/2}}
$$
with some $c>0$. Now we will prove that $A\in \Gamma_{12,3}$. Assume, to the contrary, that we have $a_1\dots a_{12}=x^3$ for some distinct elements $a_1,\dots ,a_{12}\in A$. 

Let $s=|\{a_1,\dots,a_{12}\}\cap C|$, we may assume that  $a_1,\dots a_s\in C$. Observe that $a_1\dots a_s=z^3$ and $a_{s+1}\dots a_{12}=u^3$ for some $u,z\in \mathbb{Z}^+$ as $a_1\dots a_s$ and $a_{s+1}\dots a_{12}$ are coprime. Clearly, $\Omega (a_1\dots a_s)=2s$. Hence we have $3\mid 2s$, that is $3\mid s$. 
If $s=0$, then there exist prime numbers $\sqrt{n}<p_1<p_2<p_3 < p_{4}\le n$ such that $\{ a_1,\dots, a_{12}\} =\{ p_i,2p_i,3p_i: 1\le i\le 4 \}$. Hence $a_1 \dots a_{12}=1296(p_1p_2p_3p_4)^3=u^3$, which is impossible. If $s=3$, then $\Omega (a_1a_2a_3)=6$, that is, the canonical form of $a_1a_2a_3$ is either $p_1^6$ or $p_1^3p_2^3$ for some primes $p_1,p_2$, which cases are both impossible. If $s=6$, then there exist prime numbers $\sqrt{n}<p_1< p_2\le n$ such that $\{ a_7,\dots, a_{12}\} =\{ p_i,2p_i,3p_i: 1\le i\le 2 \}$. Hence $a_7 \dots a_{12}=36(p_1p_2)^3=u^3$, which is impossible. If $s=9$, then there exist prime number $\sqrt{n}<p_1\le n$ such that $\{ a_{10},a_{11}, a_{12}\} =\{ p_1,2p_1,3p_1\}$, that is $a_{10}a_{11}a_{12}=6p_1^3=u^3$, a contradiction. If $s=12$, then $a_1\dots a_s=z^3$, we get a contradiction by Lemma~\ref{Gammak=6-l1}. 

The upper bounds are direct consequences of Lemma~\ref{Verstraete}.

(2) First, we prove the lower bound. Let
\begin{equation*}
\begin{split}
B=&\{ p:\sqrt{n}<p\le n,\ p \hbox{ prime}\} \cup \{ 2p:\sqrt{n}<p\le \frac{n}{2},\ p \hbox{ prime}\} \cup\\
&\cup \{ 3p:\sqrt{n}<p\le \frac{n}{3},\ p \hbox{ prime}\} \cup \{ 5p:\sqrt{n}<p\le \frac{n}{5},\ p \hbox{ prime}\} .
\end{split}
\end{equation*}
By the Prime Number Theorem,
$$
|B|=\pi (n)+\pi \left(\frac{n}{2}\right)+\pi \left(\frac{n}{3}\right)+\pi \left(\frac{n}{5}\right)-O\left(\frac{\sqrt{n}}{\log n}\right).
$$
Let $P=\{ p_1,p_2,\dots ,p_t\} $ denote the set of primes $p$ with $5<p\le \sqrt{n}$, by the Prime Number Theorem $t=(2+o(1))\frac{\sqrt{n}}{\log n}$. Let $G$ be a $\{C_3,C_4,C_5\}$-free graph on vertex set $\{P_1,\dots ,P_t\}$ with the maximal possible number of edges. Let $C$ denote the set of elements of the form $p_ip_j$ (where $1\le i<j\le t$) such that $p_ip_j\in C$ if and only if  $P_iP_j$ is an edge in $G_t$. By Lemma~\ref{exgraph} (1), we have 
$$
|C|\gg \frac{n^{\frac{3}{4}}}{(\log n)^{\frac{3}{2}}}.
$$
Thus, for the size of $A=B\cup C$ we have the lower bound
$$
|A|=|B|+|C|\ge \pi (n)+\pi \left(\frac{n}{2}\right) +\pi \left(\frac{n}{3}\right)+\pi \left(\frac{n}{5}\right)+c\frac{n^{\frac{3}{4}}}{(\log n)^{\frac{3}{2}}}
$$
for some $c>0$. 
Now we will prove that $A\in \Gamma_{15,3}$. Assume, to the contrary, that we have $a_1\dots a_{15}=x^3$ for some distinct elements $a_1,\dots ,a_{15}\in A$.

Let $s=|\{a_1,\dots,a_{15}\}\cap C|$, we may assume that  $a_1,\dots a_s\in C$. Observe that $a_1\dots a_s=z^3$ and $a_{s+1}\dots a_{15}=u^3$ for some $u,z\in \mathbb{Z}^+$ as $a_1\dots a_s$ and $a_{s+1}\dots a_{15}$ are coprime.
 Since $\Omega (a_1\dots a_s)=2s$,  we have $3\mid 2s$, that is $3\mid s$. 
 
 If $s=0$, then there exist prime numbers $\sqrt{n}<p_1<\dots <p_{5}\le n$ such that $$
a_1\dots a_{15}=2^{\alpha}3^{\beta}5^{\gamma}(p_1\dots p_5)^3, 
$$
where $3\mid \alpha, \beta, \gamma$ and $\alpha, \beta, \gamma \le 5$. Hence $\alpha +\beta +\gamma \le 9 $, so from the set $\{p_1,p_2,p_3,p_4,p_5\}$ at least $15-9=6$ distinct elements have to be chosen, which is impossible. If $s=3$, then $\Omega (a_1a_2a_3)=6$, that is, the canonical form of $a_1a_2a_3$ is either $p_1^6$ or $p_1^3p_2^3$ for some primes $p_1,p_2$, which cases are both impossible. If $s=6$, then there exist prime numbers $5<p_1<p_2<p_3<p_4$ such that $\{ a_1,a_2,a_3,a_4,a_5,a_6\}=\{ p_1p_2,p_1p_3,p_1p_4,p_2p_3,p_2p_4,p_3p_4\}$, that is $G$ contains a cycle $C_4$, which is impossible. If $s=9$, then there exists a $p\in \{ 2,3,5\} $ and $\alpha \in \{ 1,2\}$ such that the canonical form of $a_{10}\dots a_{15}=u^3$ contains $p^{\alpha}$, which is impossible. If $s=12$ or $s=15$, then $a_1\dots a_s=z^3$, we get a contradiction by Lemma~\ref{Gammak=6-l1}. 

The upper bound is a little modification of Verstra\"{e}te's bound (see  Lemma~\ref{Verstraete}). Let $A\subseteq [n]$, $A\in \Gamma _{15,3}$. Let $\ell(a)$ be the largest prime factor of $a\in A$. Verstra\"{e}te proved that 
$$
\left|\left\{ a: a\in A,\ \ell(a)\le \frac{n}{5} \right\}\right|\le 4\pi \left(\frac{n}{5}\right)+O(n^{5/6}).
$$
Clearly,
$$
\left|\left\{ a: a\in A,\ \frac{n}{2}<\ell(a)\le n\right\}\right|\le \pi(n)-\pi \left(\frac{n}{2}\right),
$$
$$
\left|\left\{ a: a\in A,\ \frac{n}{3}<\ell(a)\le \frac{n}{2}\right\}\right|\le 2\left(\pi\left(\frac{n}{2}\right)-\pi \left(\frac{n}{3}\right)\right),
$$
$$
\left|\left\{ a: a\in A,\ \frac{n}{4}<\ell(a)\le \frac{n}{3}\right\}\right|\le 3\left(\pi\left(\frac{n}{3}\right)-\pi \left(\frac{n}{4}\right)\right).
$$
Observe that
$$
\left|\left\{ a: a\in A,\ \frac{n}{5}<\ell(a)\le \frac{n}{4}\right\}\right|\le 3\left(\pi\left(\frac{n}{4}\right)-\pi \left(\frac{n}{5}\right)\right)+4,
$$
since otherwise there would exist prime numbers $p_i\in \left(\frac{n}{5},\frac{n}{4}\right]$, $1\le i\le 5$ such that 
$$
\{ p_i,2p_i,4p_i\} \subseteq A,
$$
and then 
$$
p_1(2p_1)(4p_1)\dots p_5(2p_5)(4p_5)=(32p_1p_2p_3p_4p_5)^3,
$$
would be a nontrivial solution. Therefore,
\begin{equation*}
\begin{split}
|A|&\le \pi(n)-\pi \left(\frac{n}{2}\right)+2\left(\pi \left(\frac{n}{2}\right)-\pi \left(\frac{n}{3}\right)\right)+\\&+3\left(\pi\left(\frac{n}{3}\right)-\pi \left(\frac{n}{4}\right)\right)+3\left(\pi\left(\frac{n}{4}\right)-\pi \left(\frac{n}{5}\right)\right)+4\pi \left(\frac{n}{5}\right)+O(n^{5/6}) \\
 &=\pi(n)+\pi\left(\frac{n}{2}\right)+\pi\left(\frac{n}{3}\right)+\pi \left(\frac{n}{5}\right)+O(n^{5/6}),
\end{split}
\end{equation*}
which completes the proof.

(3) First we prove the lower bound. Let
$$
B=\{ p:\sqrt{n}<p\le n,\ p \hbox{ prime}\} \cup \{ 2p:\sqrt{n}<p\le \frac{n}{2},\ p \hbox{ prime}\}.
$$
Then, by the Prime Number Theorem,
$$
|B|=\pi (n)+\pi \left(\frac{n}{2}\right)-O\left(\frac{\sqrt{n}}{\log n}\right).
$$
Let $P=\{ p_1,p_2,\dots ,p_t\} $ denote the set of primes $p$ with $2<p\le \sqrt{n}$, so that by the Prime Number Theorem $t=(2+o(1))\frac{\sqrt{n}}{\log n}$. Let $G_t$ be a $\{C_3,C_4,C_5,C_6\}$-free graph on the vertex set $\{P_1,\dots ,P_t\}$ with the maximal number of edges. Let $C$ denote the set of elements of the form $p_ip_j$ (where $1\le i<j\le t$) such that $p_ip_j\in C$ if and only if  $P_iP_j$ is an edge in $G_t$. By Lemma~\ref{exgraph}~(2) we have $$
|C|\gg \frac{n^{2/3}}{(\log n)^{4/3}}.
$$
Thus, for the size of $A=B\cup C$ we have the lower bound
$$
|A|=|B|+|C|\ge \pi (n)+\pi 
\left(\frac{n}{2}\right)+c\frac{n^{2/3}}{(\log n)^{4/3}}
$$
for some $c>0$. 
Now we will prove that $A\in \Gamma_{18,3}$. Assume, to the contrary, that we have $a_1\dots a_{18}=x^3$ for some $a_1,\dots ,a_{18}\in A$ such that $a_1<\dots <a_{18}$. 

Assume that $q$ is a prime with $q>\sqrt{n}$ and $q\mid x^3$. Then $q^3\mid x^3$, which is impossible, since at most two multiples of $q$ are contained in $A$. Therefore, if $p$ is a prime number with $p\mid x^3$, then $p\le \sqrt{n}$, that is, $a_i\in C$ for all $1\le i\le 18$. Hence, $a_1\dots a_{18}=z^3$, which is impossible by Lemma~\ref{Gammak=6-l1}. 

The upper bound is a direct consequence of Lemma~\ref{Verstraete}.
\end{proof}

\begin{proof}[Proof of Theorem~\ref{thm-k=3l-f}] First we prove the upper bounds.  It suffices to prove that the inequality $f_{3m+3,3}(n)\leq f_{3m,3}(n)$ for $m\ge 3$ and use the upper bound from Theorem~\ref{thm-k=9-f}. Let us assume that $A\subseteq [n]$ and $|A|>f_{3m,3}(n)$. Then there is a nontrivial solution $a_1a_2\dots a_{3m}=x^3$ (where $a_i\in A$ for each $i$). However, this implies that $a_1^4a_2\dots a_{3m}=(xa_1)^3$ is also a nontrivial solution, therefore $A\notin \gamma_{3m+3,3}$. Hence $f_{3m+3,3}(n)\le f_{3m,3}(n)$.  

The lower bounds are consequences of Lemma~\ref{exgraph} (4) and (5), Lemma~\ref{k=6-l1} and Prime Number Theorem.

%The lower bound can be proved in the same way as the lower bound in Theorem~\ref{thm-k=9-f} except that Lemma~\ref{exgraph} (2) has to be replaced by Lemma~\ref{exgraph} (3).
 \end{proof}

\begin{proof}[Proof of Theorem~\ref{thm-k=36-F}]
The proofs of the last nine inequalities are very similar to the proofs of the first three, so only the first three are proved. 

(1) First we prove the lower bound. Let
$$
B=\{ p:\sqrt{n}<p\le n,\ p \hbox{ prime}\} \cup \{ 2p:\sqrt{n}<p\le \frac{n}{2},\ p \hbox{ prime}\}.
$$
Then, by the Prime Number Theorem,
$$
|B|=\pi (n)+\pi \left(\frac{n}{2}\right)-O\left(\frac{\sqrt{n}}{\log n}\right).
$$
Let $P=\{ p_1,p_2,\dots ,p_t\} $ denote the set of primes $p$ with $2<p\le \sqrt{n}$, so that by the Prime Number Theorem $t=(2+o(1))\frac{\sqrt{n}}{\log n}$. Let $G_t$ be a $\{C_3,\dots,C_{12\ell}\}$-free graph on the vertex set $\{P_1,\dots ,P_t\}$ with the maximal number of edges. Let $C$ denote the set of elements of the form $p_ip_j$ (where $1\le i<j\le t$) such that $p_ip_j\in C$ if and only if  $P_iP_j$ is an edge in $G_t$. By Lemma~\ref{exgraph}~(4) we have $$
|C|\gg \frac{n^{\frac{9\ell}{18\ell-2}}}{(\log n)^{\frac{9\ell}{9\ell-1}}}.
$$
Thus, for the size of $A=B\cup C$ we have the lower bound
$$
|A|=|B|+|C|\ge \pi (n)+\pi 
\left(\frac{n}{2}\right)+c\frac{n^{\frac{9\ell}{18\ell-2}}}{(\log n)^{\frac{9\ell}{9\ell-1}}}
$$
for some $c>0$. 
Now we will prove that $A\in \Gamma_{36\ell,3}$. Assume, to the contrary, that we have $a_1\dots a_{36\ell}=x^3$ for some $a_1,\dots ,a_{36\ell}\in A$ such that $a_1<\dots <a_{36\ell}$. 

Assume that $q$ is a prime with $q>\sqrt{n}$ and $q\mid x^3$. Then $q^3\mid x^3$, which is impossible, since at most two multiples of $q$ are contained in $A$. Therefore, if $p$ is a prime number with $p\mid x^3$, then $p\le \sqrt{n}$, that is, $a_i\in C$ for all $1\le i\le 36\ell$. Hence, $a_1\dots a_{36\ell}=x^3$, which is impossible by Lemma~\ref{Gammak=6-l1}.

The upper bound is a direct consequence of Lemma~\ref{Verstraete}.

(2) First we prove the lower bound. Let
\begin{equation*}
    \begin{split}
        B= & \{ p:\sqrt{n}<p\le n,\  p \hbox{ prime}\} \cup \{ 2p:\sqrt{n}<p\le \frac{n}{2},\ p \hbox{ prime}\}\cup\\
        &\cup \{ 3p:\sqrt{n}<p\le \frac{n}{3},\ p \hbox{ prime}\}.
    \end{split}
\end{equation*}

By the Prime Number Theorem,
$$
|B|=\pi (n)+\pi \left(\frac{n}{2}\right)+\pi \left(\frac{n}{3}\right)-O\left(\frac{\sqrt{n}}{\log n}\right).
$$
Let $P=\{ p_1,p_2,\dots ,p_t\} $ denote the set of primes $p$ with $3<p\le \sqrt{n}$, by the Prime Number Theorem $t=(2+o(2))\frac{\sqrt{n}}{\log n}$. Let $G_t$ denote a $\{C_3,\dots,C_{12\ell+1}\} $-free graph on vertex set $\{P_1,\dots ,P_t\}$ with the maximal number of edges. Let $C$ denote the set of elements of the form $p_ip_j$ (where $1\le i<j\le t$) such that $p_ip_j\in C$ if and only if $P_iP_j$ is an edge in $G_t$. By Lemma~\ref{exgraph} (4) we have $$
|C|\gg \frac{n^{\frac{9\ell}{18\ell-2}}}{(\log n)^{\frac{9\ell}{9\ell-1}}}.
$$
Thus, for the size of $A=B\cup C$ we have the lower bound
$$
|A|=|B|+|C|\ge \pi (n)+\pi \left(\frac{n}{2}\right)+\pi \left(\frac{n}{3}\right)+c\frac{n^{\frac{9\ell}{18\ell-2}}}{(\log n)^{\frac{9\ell}{9\ell-1}}}
$$
with some $c>0$. Now we will prove that $A\in \Gamma_{36\ell+3,3}$. Assume, to the contrary, that we have $a_1\dots a_{36\ell+3}=x^3$ for some distinct elements $a_1,\dots ,a_{36\ell+3}\in A$. 

Let $s=|\{a_1,\dots,a_{36\ell+3}\}\cap C|$, we may assume that  $a_1,\dots, a_s\in C$. Observe that $a_1\dots a_s=z^3$ and $a_{s+1}\dots a_{36\ell +3}=u^3$ for some $u, z\in \mathbb{Z}^+$ as $a_1\dots a_s$ and $a_{s+1}\dots a_{36\ell+3}$ are coprime. Clearly, $\Omega (a_1\dots a_s)=2s$. Hence we have $3\mid 2s$, that is, $3\mid s$. 
If $s=0$, then there exist prime numbers $\sqrt{n}<p_1<\dots < p_{12\ell+1}\le n$ such that $\{ a_1,\dots, a_{36\ell+3}\} =\{ p_i,2p_i,3p_i: 1\le i\le 12\ell+1 \}$. Hence $a_1 \dots a_{36\ell+3}=6^{12\ell+1}(p_1\dots p_{12\ell+1})^3=u^3$, which is impossible. If $s=3$, then $\Omega (a_1a_2a_3)=6$, that is, the canonical form of $a_1a_2a_3$ is either $p_1^6$ or $p_1^3p_2^3$ for some primes $p_1,p_2$, which cases are both impossible. If $s=6$ then there exist prime numbers $\sqrt{n}<p_1<\dots < p_{12\ell-1}\le n$ such that $\{ a_7,\dots, a_{36\ell+3}\} =\{ p_i,2p_i,3p_i: 1\le i\le 12\ell-1 \}$. Hence $a_7 \dots a_{36\ell+3}=6^{12\ell-1}(p_1\dots p_{12\ell-1})^3=u^3$, which is impossible. If $s=9$ ,then there exist prime numbers $\sqrt{n}<p_1<\dots < p_{12\ell-2}\le n$ such that $\{ a_1,\dots, a_{36\ell-6}\} =\{ p_i,2p_i,3p_i: 1\le i\le 12\ell-2 \}$. Hence $a_1 \dots a_{36\ell-6}=6^{12\ell-2}(p_1\dots p_{12\ell-2})^3=u^3$, which is impossible. If $s\ge 12$, then $a_1\dots a_s=z^3$, we get a contradiction by Lemma~\ref{Gammak=6-l1}.

The upper bound is a direct consequence of Lemma~\ref{Verstraete}.

(3) Firt we prove the lower bound. Let
\begin{equation*}
    \begin{split}
        B= & \{ p:\sqrt{n}<p\le n,\  p \hbox{ prime}\} \cup \{ 2p:\sqrt{n}<p\le \frac{n}{2},\ p \hbox{ prime}\}\cup\\
        &\cup \{ 3p:\sqrt{n}<p\le \frac{n}{3},\ p \hbox{ prime}\}.
    \end{split}
\end{equation*}

By the Prime Number Theorem,
$$
|B|=\pi (n)+\pi \left(\frac{n}{2}\right)+\pi \left(\frac{n}{3}\right)-O\left(\frac{\sqrt{n}}{\log n}\right).
$$
Let $P=\{ p_1,p_2,\dots ,p_t\} $ denote the set of primes $p$ with $3<p\le \sqrt{n}$, by the Prime Number Theorem $t=(2+o(2))\frac{\sqrt{n}}{\log n}$. Let $G_t$ denote a $\{C_3,\dots,C_{12\ell+2}\} $-free graph on vertex set $\{P_1,\dots ,P_t\}$ with the maximal number of edges. Let $C$ denote the set of elements of the form $p_ip_j$ (where $1\le i<j\le t$) such that $p_ip_j\in C$ if and only if $P_iP_j$ is an edge in $G_t$. By Lemma~\ref{exgraph} (5) we have $$
|C|\gg \frac{n^{\frac{9\ell+1}{18\ell}}}{(\log n)^{\frac{9\ell+1}{9\ell}}}.
$$
Thus, for the size of $A=B\cup C$ we have the lower bound
$$
|A|=|B|+|C|\ge \pi (n)+\pi \left(\frac{n}{2}\right)+\pi \left(\frac{n}{3}\right)+c\frac{n^{\frac{9\ell+1}{18\ell}}}{(\log n)^{\frac{9\ell+1}{9\ell}}}
$$
with some $c>0$. Now we will prove that $A\in \Gamma_{36\ell+6,3}$. Assume, to the contrary, that we have $a_1\dots a_{36\ell+6}=x^3$ for some distinct elements $a_1,\dots ,a_{36\ell +6}\in A$. 

Let $s=|\{a_1,\dots,a_{36\ell+6}\}\cap C|$, we may assume that  $a_1,\dots, a_s\in C$. Observe that $a_1\dots a_s=z^3$ and $a_{s+1}\dots a_{36\ell +6}=u^3$ for some $u,z\in \mathbb{Z}^+$ as $a_1\dots a_s$ and $a_{s+1}\dots a_{18\ell+3}$ are coprime. Clearly, $\Omega (a_1\dots a_s)=2s$. Hence we have $3\mid 2s$, that is, $3\mid s$. 
If $s=0$, then there exist prime numbers $\sqrt{n}<p_1<\dots < p_{12\ell+2}\le n$ such that $\{ a_1,\dots, a_{36\ell+6}\} =\{ p_i,2p_i,3p_i: 1\le i\le 12\ell+2 \}$. Hence $a_1 \dots a_{36\ell+6}=6^{12\ell+2}(p_1\dots p_{12\ell +2})^3=x^3$, which is impossible. If $s=3$, then $\Omega (a_1a_2a_3)=6$, that is, the canonical form of $a_1a_2a_3$ is either $p_1^6$ or $p_1^3p_2^3$ for some primes $p_1,p_2$, which cases are both impossible. If $s=6$, then there exist prime numbers $3<p_1<p_2<p_3<p_4\le \sqrt{n}$ such that $\{ a_1,a_2,a_3,a_4,a_5,a_6\}=\{ p_1p_2,p_1p_3,p_1p_4,p_2p_3,p_2p_4,p_3p_4 \}$, therefore $G$ contains a $C_4$, which is impossible. If $s=9$, then there exist prime numbers $\sqrt{n}<p_1<\dots < p_{12\ell-1}\le n$ such that $\{ a_1,\dots, a_{36\ell-3}\} =\{ p_i,2p_i,3p_i: 1\le i\le 12\ell-1 \}$. Hence $a_{10} \dots a_{36\ell+6}=6^{12\ell-1}(p_1\dots p_{12\ell -1})^3=u^3$, which is impossible. If $s\ge 12$, then $a_1\dots a_s=z^3$, we get a contradiction by Lemma~\ref{Gammak=6-l1}.

The upper bounds are direct consequences of Lemma~\ref{Verstraete}.

\end{proof}

\section{Concluding remarks and open problems}
In this paper we gave bounds for the functions $F_{k,3}(n)$ and $f_{k,3}(n)$ for every $k$.

Finally, we pose some problems for further research. Clearly, $F_{1,d}(n)=f_{1,d}(n)=n-n^{1/d}+O(1)$. For $1<k<d$, the set
$$
\{ 1,2,\dots ,n\} \setminus \{ m: m \in \{ 1,2,\dots ,n\} , \exists \ell\le n^{k/d}, \ell\in \mathbb{Z}^+ \hbox{ s.t. } m\mid \ell^d  \}
$$
shows that
$$
n-F_{k,d}(n)\le \sum_{\ell\le n^{k/d}}d(\ell^d)\le n^{k/d+\varepsilon }.
$$
\bigskip

\begin{problem}
Let us suppose that $1<k<d$. Is it true that
$$
n^{k/d}\ll n-F_{k,d}(n)\le n-f_{k,d}(n)\ll n^{k/d}?
$$
\end{problem}

\begin{problem}
Is it true that there exists a constant $c$ such that
$$
f_{2,3}(n)=n-(c+o(1))n^{2/3}?
$$
\end{problem}

\begin{problem}
Let $d\ge 4$. Is it true that 
$$
f_{d+1,d}(n)=(1-o(1))n?
$$
\end{problem}

As a corollary of the above theorems we get the following result:
\begin{corollary}
For $d=2,3$ and $k>d$, $d\mid k$, there exist constants $c_{k,d}>0$ and $C_{k,d}\in \mathbb{Z}^+$ such that
$$F_{k,d}(n)=(c_{k,d}+o(1))\pi _{C_{k,d}}(n).$$
\end{corollary}

\begin{problem}
Is it true that for any $d\ge 4$ and $k>d$, $d\mid k$, there exist constants $c_{k,d}>0$ and $C_{k,d}\in \mathbb{Z}^+$ such that
$$F_{k,d}(n)=(c_{k,d}+o(1))\pi _{C_{k,d}}(n)?$$
\end{problem}

\section*{Acknowledgements.} The research was supported by the Lend\"ulet program of the Hungarian Academy of Sciences (MTA). PPP and CS were also supported by the National Research, Development and Innovation Office NKFIH (Grant Nr. K146387). CS was supported by the grant NKFI KKP144059 ”Fractal ´ geometry and applications. The first three authors would like to thank the Budapest REU 2023 program. The authors would like to thank Rich\'ard Palincza for providing help in the computer calculations needed to find the values listed in the Appendices.

\bigskip

\appendix

\section{The approximation of $c_{3,3}$}\label{table1}

The following table shows the values of $s_4(i)$ we calculated by computer and used in the approximation of $c_{3,3}$.

\bigskip
\renewcommand{\arraystretch}{1.5}
\begin{center}
\begin{tabular}{ | c |c| } 
 \hline
 interval of $a'$ & $s_4(\lfloor n/a'\rfloor)$ \\
 \hline\hline
 $a' \in (n/2, n]$ & 1  \\ 
 \hline
 $a' \in (n/3, n/2]$ & 2  \\ 
 \hline
 $a' \in (n/5, n/3]$ & 3  \\ 
 \hline
 $a' \in (n/6, n/5]$ & 4  \\ 
 \hline
 $a' \in (n/7, n/6]$ & 5  \\
 \hline
 $a' \in (n/10, n/7]$ & 6  \\
 \hline
 $a' \in (n/14, n/10]$ & 7  \\ 
 \hline
 $a' \in (n/15, n/14]$ & 8  \\ 
 \hline
 $a' \in (n/21, n/15]$ & 9  \\ 
 \hline
 $a' \in (n/25, n/21]$ & 10  \\ 
 \hline
\end{tabular}
 \quad
 \begin{tabular}{ | c |c| } 
 \hline 
  interval of $a'$ & $s_4(\lfloor n/a' \rfloor)$ \\
 \hline\hline
  $a' \in (n/30, n/25]$ & 11  \\ 
 \hline
  $a' \in (n/35, n/30]$ & 12  \\ 
 \hline
  $a' \in (n/42, n/35]$ & 13  \\ 
 \hline
  $a' \in (n/60, n/42]$ & 14  \\ 
 \hline
  $a' \in (n/70, n/60]$ & 15  \\ 
 \hline
  $a' \in (n/105, n/70]$ & 16  \\ 
 \hline
  $a' \in (n/175, n/105]$ & 17  \\ 
 \hline
  $a' \in (n/210, n/175]$ & 18  \\ 
 \hline
  $a' \in (n/315, n/210]$ & 19  \\ 
 \hline
 $a' \leq n/315$ & 20  \\ 
 \hline
\end{tabular}
\bigskip
%\caption{Approximation for $c_{3,3}$}
\end{center}
\renewcommand{\arraystretch}{1}
%\end{table}

\section{The approximation of $C_{3,3}$}\label{table2}

The following table shows the values of $S_4(i)$ we used in the approximation of $C_{3,3}$.

\bigskip
\renewcommand{\arraystretch}{1.5}
\begin{center}
\begin{tabular}{ | c |c| } 
 \hline
 interval of $a'$ & $S_4(\lfloor n/a'\rfloor)$ \\
 \hline\hline
 $a' \in (n/2, n]$ & 1  \\ 
 \hline
 $a' \in (n/3, n/2]$ & 2  \\ 
 \hline
 $a' \in (n/5, n/3]$ & 3  \\ 
 \hline
 $a' \in (n/6, n/5]$ & 4  \\ 
 \hline
 $a' \in (n/7, n/6]$ & 5  \\
 \hline
 $a' \in (n/8, n/7]$ & 6  \\
 \hline
 $a' \in (n/10, n/8]$ & 7  \\ 
 \hline
 $a' \in (n/14, n/10]$ & 8  \\ 
 \hline
 $a' \in (n/15, n/14]$ & 9  \\ 
 \hline
 $a' \in (n/16, n/15]$ & 10  \\ 
 \hline
  $a' \in (n/21, n/16]$ & 11  \\ 
 \hline
  $a' \in (n/24, n/21]$ & 12  \\ 
 \hline
  $a' \in (n/30, n/24]$ & 13  \\ 
 \hline
  $a' \in (n/35, n/30]$ & 14  \\ 
 \hline
  $a' \in (n/40, n/35]$ & 15  \\ 
 \hline
  $a' \in (n/42, n/40]$ & 16  \\ 
 \hline
  $a' \in (n/48, n/42]$ & 17  \\ 
 \hline
  $a' \in (n/56, n/48]$ & 18  \\ 
 \hline
  $a' \in (n/60, n/56]$ & 19  \\ 
 \hline
  $a' \in (n/70, n/60]$ & 20  \\ 
 \hline
\end{tabular}
 \quad
 \begin{tabular}{ | c |c| } 
 \hline 
  interval of $a'$ & $S_4(\lfloor n/a' \rfloor)$ \\
 \hline\hline
  $a' \in (n/80, n/70]$ & 21  \\ 
 \hline
  $a' \in (n/98, n/80]$ & 22  \\ 
 \hline
  $a' \in (n/105, n/98]$ & 23  \\ 
 \hline
  $a' \in (n/120, n/105]$ & 24  \\ 
 \hline
  $a' \in (n/140, n/120]$ & 25  \\ 
 \hline
  $a' \in (n/168, n/140]$ & 26  \\ 
 \hline
  $a' \in (n/200, n/168]$ & 27  \\ 
 \hline
  $a' \in (n/210, n/200]$ & 28  \\ 
 \hline
  $a' \in (n/240, n/210]$ & 29  \\ 
 \hline
  $a' \in (n/280, n/240]$ & 30  \\ 
 \hline
  $a' \in (n/392, n/280]$ & 31  \\ 
 \hline
  $a' \in (n/480, n/392]$ & 32  \\ 
 \hline
  $a' \in (n/525, n/480]$ & 33  \\ 
 \hline
  $a' \in (n/560, n/525]$ & 34  \\ 
 \hline
  $a' \in (n/784, n/560]$ & 35  \\ 
 \hline
  $a' \in (n/840, n/784]$ & 36  \\ 
 \hline
  $a' \in (n/1400, n/840]$ & 37  \\ 
 \hline
  $a' \in (n/1680, n/1400]$ & 38  \\ 
 \hline
  $a' \in (n/2520, n/1680]$ & 39  \\ 
 \hline
  $a' \leq n/2520$ & 40  \\ 
 \hline
\end{tabular}
\bigskip
%\caption{Approximation for $C_{3,3}$}
\end{center}
\renewcommand{\arraystretch}{1}
%\end{table}


\begin{thebibliography}{99}


\bibitem{Benson} C. T. Benson,
\textit{Minimal Regular Graphs of Girths Eight and Twelve}, Canadian Journal of Mathematics, \textbf{18} (1966) 1091--1094.

\bibitem{BS} J. A. Bondy, M. Simonovits, \textit{Cycles of even length in graphs}, J. Combin. theory, Ser B, \textbf{16} (1974), 97--105. 

\bibitem{Brown} W. Brown,
\textit{On Graphs that do not Contain a Thomsen Graph}, Canadian Mathematical Bulletin, \textbf{9} (3) (1966) 281--285. 

%\bibitem{Chu} F. Chung, L. Lu, 
%\textit{An Upper Bound for the Tur\'an Number $t_3(n,4)$}, Journal of Combinatorial Theory, Series A, \textbf{87} (1999) 381--389.

\bibitem{Caen3} D. de Caen, 
\textit{Extensions of a theorem of Moon and Moser on complete subgraphs}, Ars Combin. \textbf{16} (1983), 5--10.

%\bibitem{Caen} D. de Caen, L. Székely, 
%\textit{On Dense Bipartite Graphs of Girth Eight and Upper Bounds for Certain Configurations in Planar Point–Line Systems}, Journal %of Combinatorial Theory, Series A, \textbf{77} (2) (1997) 268--278. 

\bibitem{Caen2} D. de Caen, L. Székely,
\textit{The maximum size of 4-and 6-cycle free bipartite graphs on m, n vertices}, Colloquia Mathematica Societatis János Bolyai, \textbf{60} (1991) 135--142.

%\bibitem{Füredi} Z. Füredi, A. Naor, J. Verstraëte,
%\textit{On the Turán number for the hexagon}, Advances in Mathematics \textbf{203} (2006) 476--496.


%\bibitem{Győri} E. Győri, S. Kensell, C. Tompkins,
%\textit{Making a C6-free graph C4-free and bipartite}, Discrete Applied Mathematics, \textbf{209} (2016) 133--136.

\bibitem{Erdos_1964} P. Erd\H{o}s,
\textit{On the multiplicative representation of integers}, Israel Journal of Mathematics, \textbf{2} (4) (1964) 251--261.

\bibitem{Erdős} P. Erd\H{o}s,
\textit{On some applications of graph theory to number theoretic problems}, Publ. Ramanujan Inst., \textbf{1} (1969) 131--136.

\bibitem{ERS} P. Erd\H{o}s, A. R\'{e}nyi, V. T. S\'{o}s, \textit{On a problem of graph theory}, Stud. Sci. Math. Hung., \textbf{1} (1966), 215--235.

\bibitem{ES} P. Erd\H{o}s, A. S\'ark\"ozy, 
\textit{On the number of prime factors of integers}, Acta Sci. Math. Szeged, \textbf{42} (1980), 237--246. 

\bibitem{ESS} P. Erd\H{o}s, A. S\'ark\"ozy, V. T. S\'os,
\textit{On Product Representations of Powers, I.}, European Journal of Combinatorics, \textbf{16} (6) (1995) 567--588.

\bibitem{Gyori} E. Gy\H{o}ri, \textit{$C_6$-free bipartite graphs and product representation of squares}, Discrete Mathematics, \textbf{165} (1997), 371--375.


\bibitem{HR} G. H. Hardy, S. Ramanujan,
\textit{The normal number of prime factors of a number $n$}, Q. J. Math. \textbf{48}, (1920), 76--92.

\bibitem{LUW} F. Lazebnik, V. A. Ustimenko, A. J. Woldar, \textit{A new series of dense graphs of high girth}, Bull. Amer. Math. Soc., (N.S.) \textbf{32} (1995) 73--79. 

\bibitem{Pach15} P. P. Pach,
\textit{Generalized multiplicative Sidon sets}, Journal of Number Theory, \textbf{157} (2015) 507--529.

\bibitem{Pach2019} P. P. Pach, \textit{An improved upper bound for the size of the multiplicative 3-Sidon sets}, Int. J. Number Theory \textbf{15} (8) (2019), 1721--1729. 

\bibitem{PachVizer} P. P. Pach, M. Vizer \textit{Improved lower bounds for multiplicative square-free sequences} Electronic Journal of Combinatorics \textbf{30} (4) (2023) Article Number P4.31.

\bibitem{Prac} K. Prachar,
\textit{Über die kleinste quadratfreie Zahl einer arithmetischen Reihe}, Monatsh. Math. \textbf{62} (1958) 173--176.

\bibitem{Sárközy} G. N. Sárközy,
\textit{Cycles in bipartite graphs and an application in number theory}, Journal of Graph Theory, \textbf{19} (3) (1995) 323--331.

\bibitem{Simonovits} M. Simonovits, \textit{Extremal graph theory}, in: Selected topics in Graph Theory 2, L. W. Beineke and R. J. Wilson (eds), Acdemic Press, London, 1983, pp. 161--200.

\bibitem{Singleton} R. Singleton, \textit{On minimal graphs of maximum even girth}, J. Combin. Theory \textbf{1} (1966) 306--332.

\bibitem{Ver} J. Verstraëte,
\textit{Product representations of polynomials}, European Journal of Combinatorics, \textbf{27} (8) (2006) 1350--1361.




\end{thebibliography}
\end{document}